\documentclass[reqno]{amsart}
\usepackage[bookmarks=false]{hyperref} 
\usepackage{amssymb}
\usepackage{amsmath}
\usepackage{amsthm} 
\usepackage{mathtools}
\usepackage{color} 
\usepackage[svgnames]{xcolor}
\usepackage[utf8]{inputenc}
\usepackage[T1]{fontenc}
\usepackage{graphicx}
\usepackage{placeins}
\usepackage{vmargin}
\usepackage{setspace}

\hypersetup{
  colorlinks=true,
  linkcolor=Blue  ,          
  citecolor=Red,        
  filecolor=Magenta,      
  urlcolor=Green,           
pdfpagemode=UseOutlines,
pdftitle={},
pdfauthor={Tin Van Phan <van-tin.phan@univ-tlse3.fr>},
pdfsubject={Infinite soliton for derivative nonlinear Schr\"odinger equation},
pdfkeywords={soliton, Schr\"oginer equation} 
}

\numberwithin{equation}{section}
\numberwithin{figure}{section}

\newtheorem{theorem}{Theorem}[section]
\newtheorem{proposition}[theorem]{Proposition}
\newtheorem{lemma}[theorem]{Lemma}

\newtheorem*{notation}{Notation}
\newtheorem*{assumption a}{Assumption A}
\newtheorem*{assumption b}{Assumption B}
\newtheorem*{assumption c}{Assumption C}

\theoremstyle{definition}
\newtheorem{definition}[theorem]{Definition}

\theoremstyle{remark}
\newtheorem{remark}[theorem]{Remark}

\DeclarePairedDelimiter{\norm}{\lVert}{\rVert}

\newcommand{\N}{\mathbb{N}}
\newcommand{\R}{\mathbb{R}}
\newcommand{\C}{\mathbb C}

\renewcommand{\leq}{\leqslant}
\renewcommand{\geq}{\geqslant}

\DeclareMathAlphabet{\mathpzc}{OT1}{pzc}{m}{it}

\renewcommand{\Im}{\mathcal I\!\mathpzc{m}}

\begin{document}

\title[Multi-solitons, multi kink-solitons]{Construction of multi-solitons and multi kink-solitons of derivative nonlinear Schr\"odinger equations}

\author[Phan Van Tin]{Phan Van Tin}

\address[Phan Van Tin]{Institut de Math\'ematiques de Toulouse ; UMR5219,
  \newline\indent
  Universit\'e de Toulouse ; CNRS,
  \newline\indent
  UPS IMT, F-31062 Toulouse Cedex 9,
  \newline\indent
  France}
\email[Phan Van Tin]{van-tin.phan@univ-tlse3.fr}


\date{\today}

\begin{abstract} 
We look for solutions to derivative nonlinear Schrodinger equations built upon solitons. We prove the existence of multi-solitons i.e. solutions behaving at large time as the sum of finite solitons. We also show that one can attach a kink at the begin of the sum of solitons i.e multi kink-solitons. Our proofs proceed by fixed point arguments around the desired profile, using Strichartz estimates. 
\end{abstract}

\maketitle
\tableofcontents

\section{Introduction}
We consider the derivative nonlinear Schr\"odinger equation:
\begin{equation}\label{gemeral}
\begin{cases}
iu_t+u_{xx}+i\alpha |u|^2u_x+i\mu u^2\overline{u_x} + f(u)=0,\\
u(0)=u_0.
\end{cases}
\end{equation}
where $\alpha,\mu\in\R$, $f: \C \rightarrow \C$ is a given function and $u$ is a complex valueed function of $(t,x)\in\R\times\R$. 

In \cite{TsFu80,TsFu81}, Tsutsumi and Fukuda used an approximation argument to prove the existence of solutions of \eqref{gemeral} in the case $\alpha=-2$, $\mu=-1$. In this case with $f=0$, Biagioni and Linares \cite{BiLi01} proved that the solution map from $H^s(\R)$ to $C([-T,T],H^s(\R))$ is not locally uniformly continuous, for $T>0$ and $s<\frac{1}{2}$. The $H^{\frac{1}{2}}$ solution in this case is global if $\norm{u_0}^2_{L^2}<2\pi$ by the work of Miao-Wu-Xu \cite{MiWuXu11}. Later, Guo and Wu \cite{GuWu17} improved this result; that is, $H^{\frac{1}{2}}$ solution is global if $\norm{u_0}^2_{L^2}<4\pi$. The Cauchy problem of \eqref{gemeral} was also studied as in \cite{Ta99}, where gauge transformation and Fourier restriction method are used to obtain local well-posedness in $H^s$, $s \geq 1/2$. In \cite{Oz96}, Ozawa studied the Cauchy problem and gave a sufficient condition of global well-posedness for \eqref{gemeral}. The proof was used gauge transformations which reduce the original equations to systems of equations without derivative nonlinearities. In \cite{HaOz92,HaOz94}, in the case $\alpha=2\mu$, Hayashi-Ozawa proved the unique global existence of solutions to \eqref{gemeral} in Sobolev spaces and in the weighted spaces with smallness on the initial data $\norm{u_0}^2_{L^2} < \frac{4\pi}{|\alpha|}$. In the case $\alpha=-2$, $\mu=-1$, $f=0$, Wu \cite{Wu13} improved the global results in \cite{HaOz92,HaOz94}. More precisely, the author proved that the solutions exist globally in time under smallness on the initial data $\norm{u_0}_{L^2} < \sqrt{2\pi}+\varepsilon_*$, where $\varepsilon_*$ is a small positive constant. Later, Wu \cite{Wu15} improved this results for larger bounded on the initial data $\norm{u_0}_{L^2} < \sqrt{4\pi}$. The proof combines a gauge transformation and conservation laws with a sharp Gagliardo-Nirenberg inequality. In \cite{FuHaIn17}, by using variational argument, Fukaya-Hayashi-Inui gave results covering the result of Wu \cite{Wu15}. The authors showed that in the case $f=0$, $\alpha=1$, $\mu=0$, the $H^1$ solutions of \eqref{gemeral} exist globally in time for the initial satisfies $\norm{u_0}^2_{L^2}<4\pi$ or $\norm{u_0}_{L^2}^2=4\pi$ and $P(u_0)<0$, where $P$ is the momentum functional which is conserved under the flow of \eqref{gemeral}. In \cite{CoKeStTaTa02}, Colliander-Keel-Staffilani-Takaoka-Tao proved by the so-called \textit{I-method} the global well posedness in $H^s(\R)$, $s>\frac{1}{2}$ of \eqref{gemeral} if $\norm{u_0}_{L^2}^2<2\pi$ (see also \cite{CoKeStTaTa01}). In the case $f=0$ and $\mu=0$, \eqref{gemeral} is a completely integrable equation. The complete integrability structure of equation was used to prove global existence of solutions in $H^{2,2}(\R)$ by \cite{JeLiPeSu20} and in $H^s(\R)$, $s>\frac{1}{2}$ by \cite{BaPe20}. 

In the case $\mu=0$ and $f(u)=b|u|^4u$, there were a lot of works on studying stability and instability of solitons of \eqref{gemeral}. The family of solitons of \eqref{gemeral} has two parameters $(\omega,c)$. In the case $b=0$, Guo and Wu \cite{GuWu95} proved that the solitons are orbitally stable when $\omega>\frac{c^2}{4}$ and $c<0$ by using the abstract theory of Grillakis-Shatah-Strauss \cite{GrShSt87,GrShSt90}. After that, Colin and Ohta \cite{CoOh06} improved this result for all $\omega>\frac{c^2}{4}$ using variational techniques. In \cite{Ohta14}, Ohta proved that for each $b>0$ there exists a unique $s^{*}=s^{*}(b)>0 \in (0,1)$ such that the soliton $u_{\omega,c}$ is orbitally stable if $-2\sqrt{\omega}<c<2s^{*}\sqrt{\omega}$ and orbitally unstable if $2s^{*}\sqrt{\omega}<c<2\sqrt{\omega}$. In the case $b<0$, the stability result is obtained in \cite{Ha20}. In the case $b=0$, Kwon-Wu \cite{KwWu18} proved a stability result of solitons in the zero mass case. Removing the effect of scaling in the stability result of this work is an open question.

Our main goal of this paper is to study the multi-solitons theory of \eqref{gemeral}.

\subsection{Multi-solitons}\label{section multi solitons}

First, we focus on studying the following special form of \eqref{gemeral}:
\begin{equation}
\label{dNLS}
iu_t+u_{xx}+i|u|^2u_x+b|u|^4u=0.
\end{equation}          

Our first goal in this paper is to study the long time behaviour of solutions of \eqref{dNLS}. More precisely, we study the multi-solitons theory of \eqref{dNLS}. The existence of multi-solitons is a step towards the proof of the \textit{soliton resolution conjecture}, which states that all global solutions of a dispersive equation behave at large times as a sum of a radiative term and solitons. The theory of multi-soliton has attracted a lot of interest. In \cite{CoDoTs15,CoTs14}, Le Coz-Li-Tsai proved existence and uniqueness of finite and infinite soliton and kink-soliton trains of classical nonlinear Schrödinger equations, using fixed point arguments around of the desired profile. Another method was introduced in \cite{MaMe06} for the simple power nonlinear Schr\"odinger equation with $L^2$-subcritical nonlinearities. The proof was established by two ingredients: uniform estimates and a compactness property. The arguments were later modified to obtain the results for $L^2$-supercritical equations \cite{CoMaMe11} and for profiles made with excited states \cite{CoCo11}. One can also cite the works on the logarithmic Schrödinger equation (logNLS) in the focusing regime in \cite{Fe21}. In \cite{ZaSh71}, the inverse scattering transform method (IST) was used to construct multi-solitons of the one dimensional cubic focusing NLS. We would like also to mention the works on the non-linear Klein-Gordon equation \cite{CoMu14} and \cite{CoMa18}, and on the stability of multi-solitons for generalized Korteweg-de Vries equations and $L^2$-subcritical nonlinear Schrödinger equations from Martel, Merle and Tsai \cite{MaMeTs02},\cite{MaMeTs06}. In \cite{CoWu18}, Le Coz-Wu proved a stability result of multi-solitons of \eqref{dNLS} in the case $b=0$. Our motivation is to prove the existence of a multi-solitons in a similar sense as in \cite{CoTs14,CoDoTs15}. The method used in \cite{CoTs14,CoDoTs15} cannot apply directly in our case. The reason is the appearance of the derivative nonlinearities. To overcome this difficulty, we use a Gauge transformation to obtain a system of Schr\"odinger equations without derivative nonlinearities. We may use Strichartz estimates and fixed point argument to construct a suitable solution of this system. This solution satisfies a relation which is proved by using the Gr\"onwall inequality and the condition on the parameters and we obtain a solution of \eqref{dNLS}. This solution satisfies the desired property.  


Consider equation \eqref{dNLS}. The soliton of equation \eqref{dNLS} is a solution of the form $R_{\omega,c}(t,x)= e^{i\omega t}\phi_{\omega,c}(x-ct)$, where $\phi_{\omega,c} \in H^1(\R)$ solves
\begin{equation}
\label{eqofphi}
-\phi_{xx}+\omega \phi+ic\phi_x-i|\phi|^2\phi_x-b|\phi|^4\phi=0, \quad x \in \R.
\end{equation}
Applying the following gauge transform to $\phi_{\omega,c}$
\[
\phi_{\omega,c}(x)=\Phi_{\omega,c}(x)\exp \left(i\frac{c}{2}x-\frac{i}{4}\int_{-\infty}^x|\Phi_{\omega,c}(y)|^2\, dy\right),
\]
it is easy to verify that $\Phi_{\omega,c}$ (see e.g \cite[Proof of Lemma 2]{CoOh06}) satisfies the following equation.
\begin{equation}\label{eqofPhi}
-\Phi_{xx}+\left(\omega-\frac{c^2}{4}\right)\Phi+\frac{c}{2}|\Phi|^2\Phi-\frac{3}{16}\gamma|\Phi|^4\Phi=0, \quad \gamma:=1+\frac{16}{3}b.
\end{equation}
The positive even solution of \eqref{eqofPhi} is explicitly obtained by: if $\gamma>0$ ($b>\frac{-3}{16}$),
\begin{equation}\label{formula Phi}
\Phi_{\omega,c}^2(x)=\left\{ \begin{matrix}
\frac{2(4\omega-c^2)}{\sqrt{c^2+\gamma (4\omega-c^2)}\cosh(\sqrt{4\omega-c^2} x)-c} & \text{ if } -2\sqrt{\omega}<c<2\sqrt{\omega},\\
\frac{4c}{(cx)^2+\gamma} & \text{ if } c=2\sqrt{\omega},
\end{matrix}
\right.
\end{equation}
and if $\gamma \leq 0$ ($b \leq -\frac{3}{16}$),
\[
\Phi_{\omega,c}^2(x)=\frac{2(4\omega-c^2)}{\sqrt{c^2+\gamma(4\omega-c^2)}\cosh(\sqrt{4\omega-c^2}x)-c} \text{ if } -2\sqrt{\omega}<c<-2s_{*}\sqrt{\omega},
\]
where $s_{*}=s_{*}(\gamma)=\sqrt{\frac{-\gamma}{1-\gamma}}$. We note that the following condition on the parameters $\gamma$ and $(\omega,c)$ is a necessary and sufficient condition for the existence of non-trivial solutions of \eqref{dNLS} vanishing at infinity (see \cite{BeLi831}):
\begin{align*}
\text{if } \gamma>0 (\Leftrightarrow b>\frac{-3}{16}) \text{ then } & -2\sqrt{\omega}<c\leq 2\sqrt{\omega},\\
\text{if } \gamma\leq 0 (\Leftrightarrow b\leq\frac{-3}{16}) \text{ then }& -2\sqrt{\omega}<c< -2s_{*}\sqrt{\omega}.
\end{align*}
Let $(c_j,\omega_j)$ satisfying for each $1 \leq j \leq K$ the condition of existence of soliton. For each $j \in \{1,2,..,K\}$, we set
\[
R_j(t,x)=e^{i\theta_j}R_{\omega_j,c_j}(t,x).
\] 
The profile of a multi-soliton is a sum of the form:
\begin{equation}\label{profileofinfinitesoliton}
R =\sum_{j=1}^{K}R_j.
\end{equation}
A solution of \eqref{dNLS} is called a multi-soliton if
\[
\norm{u(t)-R(t)}_{H^1} \rightarrow 0 \text{ as } t \rightarrow \infty,
\]
For convenience, we set $h_j=\sqrt{4\omega_j-c_j^2}$. We rewrite
\begin{equation}
\label{soliton profile}
\Phi_{\omega_j,c_j}(x)=\sqrt{2}h_j\left(\sqrt{c_j^2+\gamma h_j^2}\cosh(h_j x)-c_j\right)^{-\frac{1}{2}}.
\end{equation}
As each soliton is in $H^{\infty}(\R)$, we have $R \in H^{\infty}(\R)$. Our first main result is the following.
\begin{theorem}\label{mainresult}
Let $K \in \N^{*}$ and for each $1 \leq j  \leq K$, let $(\theta_j,c_j,\omega_j)$ be a set of parameters such that $\theta_j \in \R$, $c_j \neq c_k$, for $j \neq k$ and $c_j$ such that $-2\sqrt{\omega_j} < c_j < 2\sqrt{\omega_j}$ if $\gamma>0$ and $-2\sqrt{\omega_j}<c_j<-2s_{*}\sqrt{\omega_j}$ if $\gamma \leq 0$. The multi-soliton profile $R$ is given as in \eqref{profileofinfinitesoliton}. Then there exists a certain positive constant $C_*$ such that if the parameters $(\omega_j,c_j)$ satisfy 
\begin{equation}\label{techiniqueestimate}
C_* \left((1+\norm{R_x}_{L_t^{\infty}L_x^{\infty}})(1+\norm{R}_{L_t^{\infty}L_x^{\infty}})+\norm{R}^4_{L_t^{\infty}L_x^{\infty}}\right) \leq v_{*}:=\mathop{\inf}\limits_{j\neq k} h_j|c_j-c_k|,
\end{equation}
then there exist $T_0>0$ depending on $\omega_1,...,\omega_K,c_1,...,c_K$ and a solution $u$ of \eqref{dNLS} on $[T_0,\infty)$ such that
\begin{equation}\label{estimateneedprove}
\norm{u-R}_{H^1} \leq C e^{-\lambda t}, \quad \forall t \geq T_0,
\end{equation}
where $\lambda=\frac{v_{*}}{16}$ and $C$ is a positive constant depending on the parameters $\omega_1,...,\omega_K,c_1,...,c_K$. 
\end{theorem}

We observe that the formula for solitons in the case $\gamma>0$ and in the case $\gamma\leq 0$ is similar. Thus, in the proof of Theorem \ref{mainresult}, we only consider the case $\gamma>0$. The case $\gamma \leq 0$ is treated by similar arguments. 

\begin{remark}\label{remark profile of multi soliton}
We give an example of parameters satisfying \eqref{techiniqueestimate}. Let $d_j <0$, $h_j\in\R$ for all $j\in \{1,2,...,K\}$ such that $d_j \neq d_k$ for all $j\neq k$. Let $(c_j,\omega_j)=\left(Md_j,\frac{1}{4}(h_j^2+M^2d_j^2)\right)$. We prove that for $M$ large enough, the condition \eqref{techiniqueestimate} is satisfied. By this choosing, we have $h_j \ll |c_j|$ and $c_j<0$ for all $j$. We have
\begin{align*}
\norm{\Phi_{\omega_j,c_j}}^2_{L^{\infty}} &\leq \frac{2h_j^2}{\sqrt{c_j^2+\gamma h_j^2}-c_j} \lesssim \frac{h_j^2}{|c_j|}.
\end{align*}
Moreover,
\begin{align*}
\partial \Phi_{\omega_j,c_j} &= \frac{-\sqrt{2}}{2}h_j^2\sqrt{c_j^2+\gamma h_j^2}\sinh(h_j x)\left(\sqrt{c_j^2+\gamma h_j^2}\cosh(h_j x)-c_j\right)^{-\frac{3}{2}}.
\end{align*}
Thus, for all $j$, we obtain
\begin{align*}
|\partial\Phi_{\omega_j,c_j}| &\lesssim h_j^2 \sqrt{c_j^2+\gamma h_j^2}|\sinh(h_j x)|\left(\sqrt{c_j^2+\gamma h_j^2}\cosh(h_j x)-c_j\right)^{-\frac{3}{2}} \\
&\lesssim h_j^2 \left(\sqrt{c_j^2+\gamma h_j^2}\cosh(h_j x)-c_j\right)^{-\frac{1}{2}}\\
&\approx h_j |\Phi_{\omega_j,c_j}| \lesssim \frac{h_j^2}{\sqrt{|c_j|}}. 
\end{align*}
In the addition, we have
\begin{align*}
\norm{\partial R_j}_{L^{\infty}} &= \norm{\partial \phi_{\omega_j,c_j}}_{L^{\infty}}\approx \norm{\partial \Phi_{\omega_j,c_j}}_{L^{\infty}}+\left\lVert\frac{c_j}{2}\Phi_{\omega_j,c_j}-\Phi^3_{\omega_j,c_j}\right\rVert_{L^{\infty}} \\ &\leq \norm{\partial \Phi_{\omega_j,c_j}}_{L^{\infty}}+\frac{|c_j|}{2}\norm{\Phi_{\omega_j,c_j}}_{L^{\infty}}+\norm{\Phi_{\omega_j,c_j}}^3_{L^{\infty}} \\
&\lesssim\frac{h_j^2}{\sqrt{|c_j|}}+h_j \sqrt{|c_j|}+\frac{h_j^3}{\sqrt{|c_j|^3}}.
\end{align*}

Thus, the left hand side of \eqref{techiniqueestimate} is bounded by 
\begin{equation}\label{technique}
C_*\left(\left(1+\sum_{1\leq j\leq K}\left(\frac{h_j^2}{\sqrt{|c_j|}}+h_j \sqrt{|c_j|}+\frac{h_j^3}{\sqrt{|c_j|^3}}\right)\right)\left(1+\sum_{1\leq j\leq K}\frac{h_j}{\sqrt{|c_j|}}\right)+\sum_{1\leq j\leq K}\frac{h_j^4}{c_j^2}\right).
\end{equation}
By our choosing, \eqref{technique} is order $M^{\frac{1}{2}}$ and the right hand side of \eqref{techiniqueestimate} is order $M^1$. Thus, \eqref{techiniqueestimate} is satisfied for $M$ large enough.
\end{remark}

\subsection{Multi kink-solitons}

Second, we consider another special case of \eqref{gemeral} as follows
\begin{equation}
\label{dNLS 2}
iu_t+u_{xx}+iu^2\overline{u_x}+b|u|^4u=0.
\end{equation}
Our goal is to construct multi kink-solitons of \eqref{dNLS 2}. The motivation comes from \cite{CoTs14,CoDoTs15}, where the authors have constructed an infinite multi kink-soliton train for classical nonlinear Schr\"odinger equations by using fixed point arguments. However, in the case of \eqref{dNLS 2}, this method can not directly be used due to the appearing of a derivative term. To overcome this difficulty, use a transformation and work on a system of two equations without derivative nonlinearites.

Consider the equation \eqref{dNLS 2}. First, we would like to define a kink solution of \eqref{dNLS 2}. Let $R_{\omega,c}$ be a smooth solution of \eqref{dNLS 2} of the form:
\begin{equation}\label{R}
R_{\omega,c}(t,x)=e^{i\omega t}\phi_{\omega,c}(x-ct), 
\end{equation}
where $\phi_{\omega,c}$ is smooth and solves
\begin{equation}\label{10}
-\phi_{xx}+\omega \phi+ic\phi_x-i\phi^2\overline{\phi_x}-b|\phi|^4\phi=0, \quad x \in \R.
\end{equation}
If $\phi_{\omega,c}\mid_{\R^{+}} \in H^1(\R^{+})$ then the following Gauge transform is well defined:
\[
\Phi_{\omega,c}= \exp \left(-i\frac{c}{2}x+\frac{i}{4}\int_{\infty}^x|\phi_{\omega,c}(y)|^2 \, dy\right)\phi_{\omega,c}.
\]
Since $\phi_{\omega,c}$ solves \eqref{10}, $\Phi_{\omega,c}$ is smooth and solves
\begin{equation}\label{eqof locolized solution}
-\Phi_{xx}+\left(\omega-\frac{c^2}{4}\right)\Phi-\frac{3}{2}\Im(\overline{\Phi}\Phi_x)\Phi-\frac{c}{2}|\Phi|^2\Phi+\frac{3}{16}\gamma|\Phi|^4\Phi=0, \quad \gamma:=\frac{5}{3}-\frac{16}{3}b.
\end{equation}
Since $\Phi_{\omega,c}\mid_{\R^{+}} \in H^2(\R^{+})$, by similar  arguments as in \cite[Proof of Lemma 2]{CoOh06}, we can prove that 
\[
\Im(\overline{\Phi_{\omega,c}}\partial_x\Phi_{\omega,c})=0.
\]
Thus, $\Phi_{\omega,c}$ solves
\begin{equation}\label{11}
-\Phi_{xx}+\left(\omega-\frac{c^2}{4}\right)\Phi-\frac{c}{2}|\Phi|^2\Phi+\frac{3}{16}\gamma|\Phi|^4\Phi=0.
\end{equation} 
Now, we give the definition of a half-kink of \eqref{dNLS}. 
\begin{definition}\label{defofkink}
The function $R_{\omega,c}$ is called a half-kink solution of \eqref{dNLS} if $R_{\omega,c}$ is of the form \eqref{R} and the associated $\Phi_{\omega,c}$ is a real valued function solving \eqref{11} and satisfying: 
\begin{equation}\label{eqofkinkprofile}
\begin{cases}
\mathop{\lim}\limits_{x \rightarrow \pm\infty}\Phi(x) \neq 0,\\
\mathop{\lim}\limits_{x \rightarrow \mp\infty}\Phi(x) =0,
\end{cases}
\end{equation}
where $\tilde\omega =\omega-\frac{c^2}{4}$, $f: \R \rightarrow \R$. 
\end{definition}
For more convenience, we define
\[
f(s)=\frac{c}{2}s^3-\frac{3}{16}\gamma s^5.
\]
The following result about the existence of a half-kink profile is stated in \cite{CoDoTs15} as follows.
\begin{proposition}\label{pro existence of half kink profile}
Let $f:\R \rightarrow \R$ be a $C^1$ function with $f(0)=0$ and define $F(s):=\int_0^s f(t)\, dt$. For $\tilde\omega \in \R$, let 
\[
\zeta(\tilde\omega):=\inf\left\{\zeta>0,F(\zeta)-\frac{1}{2}\tilde\omega \zeta^2=0\right\}
\]
and assume that there exists $\tilde\omega_1 \in \R$ such that
\begin{equation}
\zeta(\tilde\omega_1)>0,\quad f'(0)-\tilde\omega_1<0, \quad f(\zeta(\tilde\omega_1))-\tilde\omega_1\zeta(\tilde\omega_1)=0. \label{eqofzeta}
\end{equation}
Then, for $\tilde\omega=\tilde\omega_1$, there exists a half-kink profile $\Phi \in C^2(\R)$ of \eqref{eqofkinkprofile} i.e $\Phi$ is unique (up to translation), positive and satisfies $\Phi'>0$ on $\R$ and the boundary conditions
\begin{equation}\label{conditionofPhi}
\mathop{\lim}\limits_{x\rightarrow -\infty}\Phi(x)=0, \quad  \mathop{\lim}\limits_{x \rightarrow \infty}\Phi(x)=\zeta(\tilde\omega_1)>0.
\end{equation}
If in addition,
\begin{equation}\label{conditionnew for existence profile}
f'(\zeta(\tilde\omega_1))-\tilde\omega_1<0,
\end{equation}
then for any $0<a<\tilde\omega_1-\max\{f'(0),f'(\zeta(\tilde\omega_1))\}$ there exists $D_a>0$ such that
\begin{equation}
\label{behaviour of Phi}
|\Phi'(x)|+|\Phi(x)1_{x<0}|+|(\zeta(\tilde\omega_1)-\Phi(x))1_{x>0}| \leq D_a e^{-a |x|} ,\quad \forall x\in\R.
\end{equation}
\end{proposition}
We have the following remarks.
\begin{remark}\label{remark kink profile}
\item[(1)] As in \cite[Remark 1.15]{CoDoTs15}, using the symmetry $x \rightarrow -x$ and Proposition \ref{pro existence of half kink profile} implies the existence and uniqueness of half-kink profile $\Phi$ satisfying
\[
\mathop{\lim}\limits_{x\rightarrow -\infty}\Phi(x)=\zeta(\tilde\omega_1)>0,\quad \mathop{\lim}\limits_{x\rightarrow\infty}\Phi(x)=0.
\]  
\item[(2)] In our case, $f(s)=\frac{c}{2}s^3-\frac{3}{16}\gamma s^5$. We may check that if $\gamma>0$, $c>0$ then there exist $\tilde\omega_1 = \frac{c^2}{4\gamma}$ and $\zeta(\tilde{\omega}_1)=\sqrt{\frac{2c}{\gamma}}$ satisfying the conditions \eqref{eqofzeta}, \eqref{conditionnew for existence profile} and the definition of the function $\zeta$. Thus, using Proposition \ref{pro existence of half kink profile}, if $\gamma>0$, $c>0$ then there exists a half-kink solution of \eqref{dNLS} and the constant $a$ in Proposition \ref{pro existence of half kink profile} satisfy
\[
0<a<\frac{c^2}{4\gamma}.
\]
\item[(3)] Consider the half-kink profile $\Phi$ of Proposition \ref{pro existence of half kink profile}. Since $\Phi$ solves \eqref{eqofkinkprofile} and satisfies \eqref{behaviour of Phi}, we have
\[
|\Phi''(x)|+|\Phi'''(x)| \leq D_a e^{-a|x|}.
\]
\end{remark}
Now, we assume $\gamma>0$. Let $K>0$, $\theta_0,\omega_0,c_0 \in \R$ be such that $2\sqrt{\omega_0}>c_0>\sqrt{2\gamma}$. For $1 \leq j \leq K$, let $(\theta_j,\omega_j,c_j) \in \R$ be such that $c_j>c_0$, $c_j\neq c_k$ for $j\neq k$, $2\sqrt{\omega_j}>c_j>2s_{*}\sqrt{\omega_j}$ for $s_{*}=\sqrt{\frac{\gamma}{1+\gamma}}$. Set $R_j=e^{i\theta_j}R_{\omega_j,c_j}$, where $R_{\omega_j,c_j} \in H^1(\R)$ is the soliton solution of \eqref{dNLS 2} with the associated profile defined in \eqref{formula Phi}. Let $\Phi_0$ be the half-kink profile given in Remark \ref{remark kink profile} (1) associated with the parameters $\omega_0,c_0$ and $R_{\omega_0,c_0}$ be the associated half-kink solution of \eqref{dNLS 2}. Set $R_0=e^{i\theta}R_{\omega_0,c_0}$. The multi kink-soliton profile of \eqref{dNLS 2} is defined as follows:
\begin{equation}\label{multi kink soliton profile}
V=R_0+\sum_{j=1}^K R_j. 
\end{equation}
Our second main result is the following.
\begin{theorem}\label{mainresult2}
Considering \eqref{dNLS 2}, we assume that $b<\frac{5}{16}$ ($\gamma>0$). Let $V$ be given as in \eqref{multi kink soliton profile}. There exists a certain positive constant $C_*$ such that if the parameters $(\omega_j,c_j)$ satisfy
\begin{equation}\label{condition tecnique of multi kink soliton}
C_*\left(\left(1+\norm{V_x}_{L_t^{\infty}L_x^{\infty}}\right)\left(1+\norm{V}_{L_t^{\infty}L_x^{\infty}}\right)+\norm{V}^4_{L_t^{\infty}L_x^{\infty}}\right) \leq v_*:=\min\left(\inf_{j \neq k}h_j|c_j-c_k|,\inf_{j\neq 0} |c_j-c_0|\right),
\end{equation}
then there exist a solution $u$ to \eqref{dNLS 2} such that
\begin{equation}\label{estimate final}
\norm{u-V}_{H^1} \leq Ce^{-\lambda t}. \quad \forall t\geq T_0,
\end{equation}
where $\lambda=\frac{v_{*}}{16}$ and $C,T_0$ are positive constants depending on the parameters $\omega_0,...,\omega_K,c_0,...,c_K$. 
\end{theorem}
We have some following discussions about the above theorem.
\begin{remark}\label{remark constant a}
\item[(1)] The condition $c_0^2>2\gamma$ in Theorem \ref{mainresult2} is a technical condition and we can remove this. The constant $a$ in Proposition \ref{pro existence of half kink profile} satisfies
\[
0<a<\frac{c_0^2}{4\gamma}.
\]   
Thus, under the condition $c_0^2>2\gamma$, we can choose $a=\frac{1}{2}$. This fact makes the proof easier and we have
\begin{equation}
\label{estimate of Phi 0}
|\Phi_0'''(x)|+|\Phi_0''(x)|+|\Phi_0'(x)|+|\Phi_0(x)1_{x>0}|+\left|\left(\sqrt{\frac{2c_0}{\gamma}}-\Phi_0(x)\right)1_{x<0}\right| \lesssim e^{-\frac{1}{2}|x|}.
\end{equation}

\item[(2)] By similar arguments as above, we can prove that there exists a half-kink solution of \eqref{dNLS} which satisfies the definition \ref{defofkink}. To our knowledge, there are no result about stability or instability of this kind of solution.
 
\item[(3)] Let $\gamma>0$. We give an example of parameters satisfying the condition \eqref{condition tecnique of multi kink soliton} of Theorem \ref{mainresult2}. As in Remark \ref{remark profile of multi soliton}, we have
\[
\Phi_{\omega_j,c_j}=\sqrt{2}h_j\left(\sqrt{c_j^2-\gamma h_j^2}\cosh(h_j x)+c_j\right)^{\frac{-1}{2}} ,\quad \forall j=1,...,K.
\]
Hence, choosing $h_j \ll c_j$, for all $j$, we have
\begin{align*}
\norm{\Phi_{\omega_j,c_j}}^2_{L^{\infty}} &\leq \frac{2h_j^2}{\sqrt{c_j^2-\gamma h_j^2}+c_j}\lesssim \frac{h_j^2}{c_j}.
\end{align*} 
By similar arguments as in Remark \ref{remark profile of multi soliton}, for all $1 \leq j \leq K$, we have
\begin{align*}
\norm{\partial R_j}_{L^{\infty}} &\lesssim \frac{h_j^2}{\sqrt{c_j}}+h_j\sqrt{c_j}+\frac{h_j^3}{\sqrt{c_j^3}}.
\end{align*}
Now, we treat to the case $j=0$. Let $\Phi_0$ be the profile given as in Proposition \ref{pro existence of half kink profile} associated to the parameters $c_0$, $\omega_0$ and $R_0$ be the associated half-kink solution of \eqref{dNLS}. From \eqref{behaviour of Phi}, Remark \ref{remark kink profile} and Remark \ref{remark constant a} we have
\begin{align*}
\norm{\Phi_0}_{L^{\infty}} &\lesssim   \sqrt{c_0},\\
\norm{\partial\Phi_0}_{L^{\infty}} &\lesssim 1,
\end{align*}
Thus,
\begin{align*}
\norm{R_0}_{L^{\infty}L^{\infty}} &\lesssim \sqrt{c_0},\\
\norm{\partial R_0}_{L^{\infty}L^{\infty}} &\lesssim 1+c_0^{\frac{3}{2}} \lesssim c_0^{\frac{3}{2}}.
\end{align*}
This implies that for $h_j \ll c_j$ ($j=1,..,K$) the left hand side of \eqref{condition tecnique of multi kink soliton} is estimated by:
\begin{align*}
C_*\left(\left(1+ c_0^{\frac{3}{2}}+\sum_{j=1}^K \left(\frac{h_j^2}{\sqrt{c_j}}+h_j\sqrt{c_j}+\frac{h_j^3}{\sqrt{c_j^3}}\right) \right)\left(1+\sqrt{c_0}+\sum_{j=1}^K\frac{h_j}{\sqrt{c_j}}\right)\right).
\end{align*}
Choosing $c_0 \approx 1$, the above expression is estimated by:
\begin{equation}
\label{eq1-8new}
C_*\left(\left(1+ \sum_{j=1}^K \left(\frac{h_j^2}{\sqrt{c_j}}+h_j\sqrt{c_j}+\frac{h_j^3}{\sqrt{c_j^3}}\right) \right)\left(1+\sum_{j=1}^K\frac{h_j}{\sqrt{c_j}}\right)\right). 
\end{equation}
Let $h_j,d_j \in \R^{+}$, $d_j \neq d_k$ for all $j\neq k$, $1\leq j,k\leq K$. Set $c_j=Md_j$, $\omega_j=\frac{1}{4}(h_j^2+M^2d_j^2)$. We have \eqref{eq1-8new} is of order $M^0$ and the right hand side of \eqref{condition tecnique of multi kink soliton} is of order $M^1$. Thus, by these choices of parameters, when $M$ is large enough, the condition \eqref{condition tecnique of multi kink soliton} is satisfied.   
\end{remark}

The proof of Theorem \ref{mainresult2} uses similar arguments as in the one of Theorem \ref{mainresult}. To prove Theorem \ref{mainresult}, our strategy is the following. Let $R$ be the multi-soliton profile. Our aim is to construct a solution of \eqref{dNLS} which behaves as $R$ at large times. Using the Gauge transform \eqref{gauge transform 1}, we construct a system of equations of $(\varphi,\psi)$. Let $h,k$ be the profile under the Gauge transform of $R$. We see that $h,k$ solves the same system as $\varphi,\psi$ up to exponential decay pertubations. The decay of these terms is showed by using the separation of solitons. Set $\tilde\varphi=\varphi-h$ and $\tilde\psi=\psi-k$. We see that if $u$ solves \eqref{dNLS} then $(\tilde\varphi,\tilde\psi)$ solves \eqref{systemneedtosolve}. By using the Banach fixed point theorem, we show that there exists a solution of this system which decays exponentially fast at infinity. Using this property and combining with the condition \eqref{techiniqueestimate}, we may prove a relation between $\tilde\varphi$ and $\tilde\psi$. This relation allows us to obtain a solution of \eqref{dNLS} satisfying the desired property.    

This chapter is organized as follows. In the section \ref{section4.2}, we prove the existence of multi-solitons for the equation \eqref{dNLS}. In the section \ref{section 4.3}, we prove the existence of multi kink-solitons for the equation \eqref{dNLS 2}. In the section \ref{chapter technique}, we prove some tools which is used in the proofs in the section \ref{section4.2} and the section \ref{section 4.3}. More precisely, we prove the exponential decay of the pertubations in the equations of $h,k$ (Lemma \ref{lemma1}, Lemma \ref{lm100}) and the existence of exponential decay solutions of the systems considered in the proofs of the main results in the section \ref{section4.2} (Lemma \ref{lemma2}).   

Before proving the main results, we recall Strichartz estimates and introduce some notations used in this chapter. We need the following definition of admissible pairs. 
\begin{definition}
Let $N\in\N^{*}$. We say that a pair $(q,r)$ is admissible if 
\[
\frac{2}{q}=N\left(\frac{1}{2}-\frac{1}{r}\right),
\]
and
\[
2\leq r \leq \frac{2N}{N-2} \quad (2\leq r\leq\infty \text{ if } N=1 \, 2\leq r <\infty \text{ if } N=2).
\]
\end{definition}

\begin{lemma} (Strichartz estimates)(see e.g \cite[Theorem 2.3.3]{Ca03})
Let $S(t)$ be the Schr\"odinger group. The following properties holds:
\begin{itemize}
\item[(i)] There exists a constant $C$ such that for all $\varphi\in L^2(\R^N)$, we have
\[
\norm{S(\cdot)\varphi}_{L^q(\R,L^r)} \leq C\norm{\varphi}_{L^2},
\]
for every admissible pair $(q,r)$.
\item[(ii)] Let $I$ be an interval of $\R$ and $t_0\in\overline{I}$. Let $(\gamma,\rho)$ be an admissible pair and $f\in L^{\gamma'}(I,L^{\rho'}(\R^N))$. Then, for all admissible pair $(q,r)$, the function
\[
t\mapsto \Phi_f(t)=\int_{t_0}^t S(t-s)f(s)\,ds
\]
belong to $L^q(I,L^r(\R^N))\cap C(I,L^2(\R^N))$. Moreover, there exists a constant $C$ independent of $I$ such that
\[
\norm{\Phi_f}_{L^q(I,L^r)}\leq C\norm{f}_{L^{\gamma'}(I,L^{\rho'})}, \quad \text{for all } f\in L^{\gamma'}(I,L^{\rho'}(\R^N)).
\]
\end{itemize}
\end{lemma}

\begin{notation}
\label{notation}

\item[(1)] For $t>0$, the Strichartz space $S([t,\infty))$ is defined via the norm
\begin{align*}
\norm{u}_{S([t,\infty))}&=\mathop{\sup}\limits_{(q,r) \text{ admissible}}\norm{u}_{L^q_{\tau}L^r_x([t,\infty) \times \R)}
\end{align*}
The dual space is denoted by $N([t,\infty))=S([t,\infty))^{*}$.

\item[(2)] For $z=(a,b) \in \C^2$ a vector, we denote $|z|=|a|+|b|$.

\item[(3)] We denote $a \lesssim b$, for $a,b>0$, if $a$ is smaller than $b$ up to multiplication by a positive constant. Moreover, we denote $a \approx b$ if $a$ equal to $b$ up to multiplication by a positive constant.   

\item[(4)] We denote $a \lesssim_k b$ if there exists a constant $C(k)$ depending only on $k$ such that $a \leq C(k)b$. \\
Particularly, we denote $a \lesssim_p b$ if there exists a constant $C$ depending only on the parameters $\omega_1,...,\omega_K,c_1,...,c_K$ such that $a \leq Cb$.

\item[(5)] Let $f \in C^1(\R)$. We use $\partial f$ or $f_x$ to denote the derivative in space of the function $f$. 

\item[(6)] Let $f(x,y,z,..)$ be a function. We denote $|df|=|f_x|+|f_y|+|f_z|+...$.
\end{notation}

\section{Proof of Theorem \ref{mainresult}}
\label{section4.2}

In this section, we give the proof of Theorem \ref{mainresult}. We divide our proof into three steps.

\textbf{Step 1. Preliminary analysis}

Considering the following transform:
\begin{equation}
\label{gauge transform 1}
\begin{cases}
\varphi(t,x)=\exp \left(\frac{i}{2}\int_{-\infty}^{x}|u(t,y)|^2\, dy\right)u(t,x),\\
\psi =\partial \varphi-\frac{i}{2}|\varphi|^2\varphi.
\end{cases}
\end{equation}
By similar arguments as in \cite{HaOz94} and \cite{Oz96}, we see that if $u$ solves \eqref{dNLS} then $(\varphi,\psi)$ solves the following system
\begin{equation}
\label{systemtransform}
\begin{cases}
L\varphi=i\varphi^2\overline{\psi}-b|\varphi|^4\varphi,\\
L\psi=-i\psi^2\overline{\varphi}-3b|\varphi|^4\psi-2b|\varphi|^2\varphi^2\overline{\psi},\\
\varphi\mid_{t=0} = \varphi_0=\exp \left(\frac{i}{2}\int_{-\infty}^x|u_0(y)|^2\, dy \right)u_0,\\
\psi\mid_{t=0}=\psi_0=\partial\varphi_0-\frac{i}{2}|\varphi_0|^2\varphi_0,
\end{cases}
\end{equation}  
where $L=i\partial_t+\partial_{xx}$. Define
\begin{align*}
P(\varphi,\psi)&=i\varphi^2\overline{\psi}-b|\varphi|^4\varphi,\\
Q(\varphi,\psi)&=-i\psi^2\overline{\varphi}-3b|\varphi|^4\psi-2b|\varphi|^2\varphi^2\overline{\psi}.
\end{align*}
Let $R$ be the multi soliton profile given in \eqref{profileofinfinitesoliton}. Since $R_j$ solves \eqref{dNLS}, for all $j$, by an elementary calculation, we have
\begin{equation}\label{eqR}
iR_t+R_{xx}+i|R|^2R_x+b|R|^4R=i\left(|R|^2R_x-\sum_{j=1}^{K}|R_j|^2R_{jx}\right)+b\left(|R|^4R-\sum_{j=1}^{K}|R_j|^4R_j\right).
\end{equation}  
From Lemma \ref{lemma1}, we have
\begin{equation}\label{eqestimate123}
\left\lVert|R|^2R_x-\sum_{j=1}^{K}|R_j|^2R_{jx}\right\rVert_{H^2} +\left\lVert|R|^4R-\sum_{j=1}^{K}|R_j|^4R_j\right\rVert_{H^2} \leq e^{-\lambda t},
\end{equation}
where $\lambda =\frac{1}{16} v_{*}$. Thus, we rewrite \eqref{eqR} as follows
\begin{equation}
\label{eqRnew}
iR_t+R_{xx}+i|R|^2R_x+b|R|^4R=e^{-\lambda t}v(t,x),
\end{equation}
where $v(t) \in H^2(\R)$ is such that $\norm{v(t)}_{H^2}$ is uniformly bounded in $t$. Define
\begin{align}
h(t,x)&= \exp\left(\frac{i}{2}\int_{-\infty}^x|R|^2\, dy\right)R(t,x), \label{eqof h}\\
k&= h_x-\frac{i}{2}|h|^2h. \label{eqof k}
\end{align} 
By an elementary calculation, we have
\begin{align*}
Lh&=ih^2\overline{k}-b|h|^4h+e^{-t\lambda}m(t,x)=P(h,k)+e^{-t\lambda}m(t,x),\\
Lk&=-ik^2\overline{h}-3b|h|^4k-2b|h|^2h^2\overline{k}+e^{-t\lambda}n(t,x)=Q(h,k)+e^{-t\lambda}n(t,x),
\end{align*}
where  $m,n$ satisfy
\begin{align}
m&=v\exp\left(\frac{i}{2}\int_{-\infty}^x|R|^2\,dy\right)-h\int_{-\infty}^x\Im(v\overline{R})\, dy \label{expressionofmn1},\\
n&=m_x-i|h|^2m+\frac{i}{2}h^2\overline{m}. \label{expressionofmn2} 
\end{align}
From Lemma \ref{lm4}, we have $\norm{m(t)}_{H^1}+\norm{n(t)}_{H^1}$ uniformly bounded in $t$. Set $\tilde \varphi =\varphi-h$ and $\tilde \psi=\psi-k$. Then $\tilde\varphi$, $\tilde\psi$ solve:
\begin{equation}
\label{systemneedtosolve}
\begin{cases}
L\tilde\varphi=P(\varphi,\psi)-P(h,k)-e^{-t\lambda}m(t,x),\\
L\tilde\psi=Q(\varphi,\psi)-Q(h,k)-e^{-t\lambda}n(t,x).
\end{cases} 
\end{equation}
Set $\eta=(\tilde\varphi,\tilde\psi)$, $W=(h,k)$, $H=-e^{-t\lambda}(m,n)$ and $f(\varphi,\psi)=(P(\varphi,\psi),Q(\varphi,\psi))$. We express solutions of \eqref{systemneedtosolve} in the following form:
\begin{equation}\label{eqof eta}
\eta(t) = i\int_t^{\infty}S(t-s)[f(W+\eta)-f(W)+H](s)\,ds,
\end{equation}
where $S(t)$ is the Schr\"odinger group. Moreover, by using $\psi=\partial \varphi-\frac{i}{2}|\varphi|^2\varphi$, we have
\begin{equation}
\label{relation}
\tilde\psi=\partial\tilde\varphi-\frac{i}{2}(|\tilde\varphi+h|^2(\tilde\varphi+h)-|h|^2h).
\end{equation}

\textbf{Step 2. Existence a solution of \eqref{systemneedtosolve}}

From Lemma \ref{lemma2}, there exists $T_{*} \gg 1 $ such that for $T_0 \geq T_{*}$ there exists a unique solution $\eta$ defined on $[T_0,\infty)$ of \eqref{systemneedtosolve} such that 
\begin{equation}\label{estimateofsolutionforsystem 1}
e^{t\lambda}(\norm{\eta}_{S([t,\infty)) \times S([t,\infty))})+e^{t \lambda}(\norm{\eta_x}_{S([t,\infty)) \times S([t,\infty))}) \leq 1, \quad \forall t\geq T_0,
\end{equation}
Thus, for all $t \geq T_0$, we have
\begin{align}\label{estimateofsolutionforsystem 2}
\norm{\tilde\varphi}_{H^1}+\norm{\tilde\psi}_{H^1} \lesssim e^{-\lambda t},
\end{align}  

\textbf{Step 3. Existence of multi-solitons} 

Let $\eta$ be the solution of \eqref{systemneedtosolve} found in step 1. We prove that the solution $\eta=(\tilde\varphi,\tilde\psi)$ of \eqref{systemneedtosolve} satisfies the relation \eqref{relation}. Set $\varphi=\tilde\varphi+h$, $\psi=\tilde\psi+k$ and 
\[
v=\partial \varphi-\frac{i}{2}|\varphi|^2\varphi.
\]
Since $h$ solves $Lh=P(h,k)+e^{-t\lambda}m(t,x)$ and $\tilde\varphi$ solves $L\tilde\varphi=P(\varphi,\psi)-P(h,k)-e^{-t\theta}m(t,x)$, we have $L\varphi=P(\varphi,\psi)$. Similarly, $L\psi=Q(\varphi,\psi)$. We have
\begin{equation*}
\begin{cases}
L\varphi=P(\varphi,\psi),\\
L\psi=Q(\varphi,\psi).
\end{cases}
\end{equation*}
Thus,
\begin{align}
L\psi-Lv&=Q(\varphi,\psi) - \left(\partial L\varphi-\frac{i}{2}L(|\varphi|^2\varphi)\right) \nonumber\\
&= Q(\varphi,\psi)-\left(\partial L\varphi - \frac{i}{2}(L(\varphi^2)\overline{\varphi}+\varphi^2L(\overline{\varphi}) +2\partial(\varphi^2)\partial\overline{\varphi})\right) \nonumber\\
&= Q(\varphi,\psi)-\left(\partial L\varphi-\frac{i}{2}(2L\varphi |\varphi|^2+2(\partial\varphi)^2\overline{\varphi}-\varphi^2\overline{L\varphi}+2\varphi^2\partial_{xx}\overline{\varphi})+4\varphi|\partial\varphi|^2)\right). \label{eqfinal}
\end{align}
Moreover,
\begin{align}
L\varphi &=P(\varphi,\psi)=i\varphi^2\overline{\psi}-b|\varphi|^4\varphi \nonumber\\
&=i\varphi^2\overline{(\psi-v)}+i\varphi^2\overline{v}-b|\varphi|^4\varphi.\label{eqoffinal}
\end{align}
Combining \eqref{eqoffinal} and \eqref{eqfinal} and by an elementary calculation, we obtain
\begin{align}
L\psi-Lv&=Q(\varphi,\psi)-\partial(i\varphi^2\overline{(\psi-v)})-|\varphi|^2\varphi^2\overline{(\psi-v)}-\frac{1}{2}|\varphi|^4(\psi-v)-Q(\varphi,v) \nonumber\\
&= (Q(\varphi,\psi)-Q(\varphi,v))-2i\varphi\partial\varphi\overline{(\psi-v)}-i\varphi^2\partial\overline{(\psi-v)}\nonumber\\
&\quad-|\varphi|^2\varphi^2\overline{(\psi-v)}-\frac{1}{2}|\varphi|^4(\psi-v)\nonumber\\
&=-i(\psi^2-v^2)\overline{\varphi}-3b|\varphi|^4(\psi-v)-2b|\varphi|^2\varphi^2\overline{(\psi-v)}\nonumber\\
&\quad -2i\varphi\left(v+\frac{i}{2}|\varphi|^2\varphi\right)\overline{(\psi-v)}-i\varphi^2\partial\overline{(\psi-v)}\nonumber\\
&\quad -|\varphi|^2\varphi^2\overline{(\psi-v)}-\frac{1}{2}|\varphi|^4(\psi-v). \label{eq}
\end{align}

Define $\tilde v= v-k$. Since $\tilde \psi-\tilde v=\psi-v$ and \eqref{eq} we have
\begin{equation}\label{eqnew}
L\tilde\psi-L\tilde v=(\tilde\psi-\tilde v)A(\tilde\psi,\tilde v,\tilde \varphi,h,k)+\overline{(\tilde\psi-\tilde v)}B(\tilde\psi,\tilde v,\tilde \varphi,h,k)-i(\tilde\varphi + h)^2\partial\overline{(\tilde\psi-\tilde v)},
\end{equation}
where 
\begin{align*}
A&= -i(\tilde \psi+\tilde v+2k)(\overline{\tilde \varphi+h})-3b|\tilde \varphi+h|^4-\frac{1}{2}|\tilde \varphi+h|^4,\\
B&=-2b|\tilde\varphi+h|^2(\tilde\varphi+h)^2-2i(\tilde\varphi+h)\left(\tilde v+k+\frac{i}{2}|\tilde\varphi+h|^2(\tilde\varphi+h)\right)-|\tilde \varphi+h|^2(\tilde\varphi+h)^2.
\end{align*}
We see that $A,B$ are polynomials of degree at most 4 in $(\tilde\psi,\tilde v,\tilde \varphi,h,k)$. Multiplying both sides of \eqref{eqnew} by $\overline{\tilde\psi-\tilde v}$ then taking imaginary part and integrating over space using integration by parts, we obtain
\begin{align*}
\frac{1}{2}\partial_t\norm{\tilde\psi-\tilde v}^2_{L^2}&= \Im\int_{\R}(\tilde\psi-\tilde v)^2 A(\tilde\psi,\tilde v,\tilde \varphi,h,k) + \overline{(\tilde\psi-\tilde v)}^2B(\tilde\psi,\tilde v,\tilde \varphi,h,k)\\
&\quad +\frac{i}{2}\partial(\tilde\varphi+ h)^2\overline{(\tilde\psi-\tilde v)}^2\, dx.
\end{align*}
Thus,
\begin{align*}
\left|\frac{1}{2}\partial_t\norm{\tilde\psi-\tilde v}^2_{L^2}\right|&\lesssim \norm{\tilde \psi-\tilde v}^2_{L^2}(\norm{A}_{L^{\infty}}+\norm{B}_{L^{\infty}}+\norm{\partial(\tilde\varphi+ h)^2}_{L^{\infty}}).
\end{align*}
By using Gr\"onwall inequality, we obtain
\begin{align}
&\norm{\tilde \psi(t)-\tilde v(t)}^2_{L^2}\nonumber\\
&\lesssim \norm{\tilde \psi(N)-\tilde v(N)}^2_{L^2} \exp \left(\int_t^N (\norm{A}_{L^{\infty}}+\norm{B}_{L^{\infty}}+\norm{\partial(\tilde\varphi+h)^2}_{L^{\infty}}\, ds\right).\label{eq1595}
\end{align}
Combining \eqref{estimateofsolutionforsystem 1}, \eqref{estimateofsolutionforsystem 2}, using $k=h_x-\frac{i}{2}|h|^2h$, $\tilde v=\partial\tilde\varphi-\frac{i}{2}(|\tilde\varphi+h|^2(\tilde\varphi+h)-|h|^2h)$, $|h|=|R|$ and the Sobolev embedding $H^1(\R) \hookrightarrow L^{\infty}$, we have, for $t \geq T_0$:
\begin{align*}
\norm{\tilde\varphi+h}_{L^{\infty}} &\lesssim 1+\norm{h}_{L^{\infty}},\\
\norm{\tilde v}_{L^{\infty}}&=\left\lVert\partial\tilde\varphi-\frac{i}{2}(|\tilde\varphi+h|^2(\tilde\varphi+h)-|h|^2h)\right\rVert_{L^{\infty}}\lesssim \norm{\partial\tilde\varphi}_{L^{\infty}}+\norm{\tilde\varphi}^3_{L^{\infty}}+\norm{\tilde\varphi}_{L^{\infty}}\norm{h}^2_{L^{\infty}}\\
&\quad\lesssim 1+\norm{\partial\tilde\varphi}_{L^{\infty}}+\norm{h}^2_{L^{\infty}}.
\end{align*}
Thus,
\begin{align*}
&\int_t^N(\norm{A}_{L^{\infty}}+\norm{B}_{L^{\infty}}+\norm{\partial(\tilde\varphi+ h)^2}_{L^{\infty}})\, ds\\
&\lesssim \int_t^N (\norm{\tilde\psi}_{L^{\infty}}+\norm{\tilde v}_{L^{\infty}}+\norm{k}_{L^{\infty}})\norm{\tilde\varphi+h}_{L^{\infty}}+\norm{\tilde{\varphi}+h}_{L^{\infty}}^4+\norm{\tilde{\varphi}+h}_{L^{\infty}}\norm{\tilde{v}+k}_{L^{\infty}}\\
&\quad +(\norm{\tilde{\varphi}}_{L^{\infty}}+\norm{h}_{L^{\infty}})(\norm{\partial\tilde{\varphi}}_{L^{\infty}}+\norm{h_x}_{L^{\infty}})\,ds\\
&\lesssim \int_t^N (1+\norm{\tilde v}_{L^{\infty}}+\norm{k}_{L^{\infty}})(1+\norm{h}_{L^{\infty}})+1+\norm{h}^4_{L^{\infty}}+(1+\norm{h}_{L^{\infty}})(\norm{\tilde{v}}_{L^{\infty}}+\norm{k}_{L^{\infty}}) \\
&+(1+\norm{h}_{L^{\infty}})(\norm{\partial\tilde{\varphi}}_{L^{\infty}}+\norm{h_x}_{L^{\infty}})\,ds\\
&\lesssim \int_t^N 1+\norm{h}^4_{L^{\infty}}+\norm{k}_{L^{\infty}}(1+\norm{h}_{L^{\infty}})+\norm{\tilde{v}}_{L^{\infty}}(1+\norm{h}_{L^{\infty}})\\
&\quad+(1+\norm{h}_{L^{\infty}})(\norm{\partial\tilde{\varphi}}_{L^{\infty}}+\norm{h_x}_{L^{\infty}}) \,ds\\
&\lesssim \int_t^N 1+\norm{h}^4_{L^{\infty}}+\norm{k}_{L^{\infty}}(1+\norm{h}_{L^{\infty}})+\norm{\partial\tilde{\varphi}}_{L^{\infty}}(1+\norm{h}_{L^{\infty}}) \\
&\quad+(1+\norm{h}_{L^{\infty}})(\norm{\partial\tilde{\varphi}}_{L^{\infty}}+\norm{k}_{L^{\infty}}+\norm{h}^3_{L^{\infty}})\,ds\\
&\lesssim \int_t^N 1+\norm{h}^4_{L^{\infty}}+\norm{k}_{L^{\infty}}(1+\norm{h}_{L^{\infty}})+\norm{\partial\tilde{\varphi}}_{L^{\infty}}(1+\norm{h}_{L^{\infty}})\,ds\\
&\lesssim (N-t)(1+\norm{h}^4_{L^{\infty}L^{\infty}}+\norm{k}_{L^{\infty}L^{\infty}}(1+\norm{h}_{L^{\infty}L^{\infty}}))\\
&\quad +\norm{\partial\tilde{\varphi}}_{L^4(t,N)L^{\infty}}(\norm{1}_{L^{\frac{4}{3}}(t,N)}+\norm{h}_{L^{\frac{4}{3}}(t,N)L^{\infty}})\\
&\lesssim (N-t)(1+\norm{R}^4_{L^{\infty}L^{\infty}}+(\norm{h_x}_{L^{\infty}L^{\infty}}+\norm{R}^3_{L^{\infty}L^{\infty}})(1+\norm{R}_{L^{\infty}L^{\infty}}))\\
&\quad + (N-t)^{\frac{3}{4}}(1+\norm{R}_{L^{\infty}L^{\infty}}^{\frac{4}{3}})\\
&\lesssim (N-t)(1+\norm{R}_{L^{\infty}L^{\infty}}^4+\norm{R_x}_{L^{\infty}L^{\infty}}(1+\norm{R}_{L^{\infty}L^{\infty}})) + (N-t)^{\frac{3}{4}}(1+\norm{R}_{L^{\infty}L^{\infty}}^{\frac{4}{3}}).
\end{align*}
Thus, there exists a certain positive constant $C_0$ such that
\begin{align*}
&\int_t^N(\norm{A}_{L^{\infty}}+\norm{B}_{L^{\infty}}+\norm{\partial(\tilde\varphi+ h)^2}_{L^{\infty}})\, ds\\
&\leq C_0\left( (N-t)(1+\norm{R}_{L^{\infty}L^{\infty}}^4+\norm{R_x}_{L^{\infty}L^{\infty}}(1+\norm{R}_{L^{\infty}L^{\infty}})) + (N-t)^{\frac{3}{4}}(1+\norm{R}_{L^{\infty}L^{\infty}}^{\frac{4}{3}})\right).
\end{align*}
Let $C_*=32C_0$. From the assumption \eqref{techiniqueestimate}, we have 
\[
C_0 \left((1+\norm{R_x}_{L^{\infty}L^{\infty}})(1+\norm{R}_{L^{\infty}L^{\infty}})+\norm{R}^4_{L^{\infty}L^{\infty}}\right) \leq  \frac{v_*}{32}=\frac{\lambda}{2}.
\] 
Hence, fix $t$ and let $N$ large enough, we have
\[
\int_t^N(\norm{A}_{L^{\infty}}+\norm{B}_{L^{\infty}}+\norm{\partial(\tilde\varphi+ h)^2}_{L^{\infty}})\, ds \leq (N-t)\lambda.
\]
Combining with \eqref{estimateofsolutionforsystem 2} and \eqref{eq1595}, we obtain, for $N$ large enough: 
\begin{align*}
\norm{\tilde \psi(t)-\tilde v(t)}^2_{L^2}  & \lesssim e^{-2\lambda N}e^{(N-t)\lambda}= e^{-\lambda N-t\lambda}.
\end{align*} 
Let $N \rightarrow \infty$, we obtain
\[
\norm{\tilde \psi(t)-\tilde v(t)}^2_{L^2}=0.
\]
This implies that $\tilde \psi=\tilde v$ and we have 
\begin{equation}\label{conservationofrelation}
\psi=v=\partial \varphi-\frac{i}{2}|\varphi|^2\varphi.
\end{equation} 
Define $u=\exp\left(-\frac{i}{2}\int_{-\infty}^x|\varphi(y)|^2\, dy\right)\varphi$. Combining \eqref{conservationofrelation} with the fact that $(\varphi,\psi)$ solves
\begin{equation*}
\begin{cases}
L\varphi=P(\varphi,\psi),\\
L\psi=Q(\varphi,\psi),
\end{cases}
\end{equation*}
we obtain that $u$ solves \eqref{dNLS}. Moreover,
\begin{align*}
\norm{u-R}_{H^1} &= \left\lVert\exp\left(-\frac{i}{2}\int_{-\infty}^x|\varphi(y)|^2\,dy\right)\varphi-\exp\left(-\frac{i}{2}\int_{-\infty}^x|h(y)|^2\,dy\right)h\right\rVert_{H^1}\\
&\lesssim \norm{\varphi-h}_{H^1}=\norm{\tilde\varphi}_{H^1}
\end{align*}
Combining with \eqref{estimateofsolutionforsystem 2}, for $t \geq T_0$, we have
\[
\norm{u-R}_{H^1} \leq C e^{-\lambda t},
\]
for a constant $C$ depending on the parameters $\omega_1,...,\omega_K,c_1,...,c_K$. This completes the proof of Theorem \ref{mainresult}.

\section{Proof of Theorem \ref{mainresult2}}\label{section 4.3}
In this section, we prove Theorem \ref{mainresult2}. We use the similar idea in the proof of Theorem \ref{mainresult}. However, the argument used in this section cannot apply to \eqref{dNLS} (see Remark \ref{rm10}). We divide our proof into three steps:

\textbf{Step 1. Preliminary analysis}

Set
\[
v:= u_x+\frac{i}{2}|u|^2u.
\]
By an elementary calculation, we see that if $u$ solves \eqref{dNLS} then $(u,v)$ solves the following system:
\begin{equation}
\begin{cases}
Lu=-iu^2\overline{v}+\left(\frac{1}{2}-b\right)|u|^4u,\\
Lv=iv^2\overline{u}+\left(\frac{3}{2}-3b\right)|u|^4v+(1-2b)|u|^2u^2\overline{v},\\
u\mid_{t=0}=u_0,\\
v\mid_{t=0}=v_0=\partial u_0+\frac{i}{2}|u_0|^2u_0.
\end{cases}
\end{equation}
Define
\begin{align*}
P(u,v)&=-iu^2\overline{v}+\left(\frac{1}{2}-b\right)|u|^4u,\\
Q(u,v)&=iv^2\overline{u}+\left(\frac{3}{2}-3b\right)|u|^4v+(1-2b)|u|^2u^2\overline{v}.
\end{align*}
Let $V$ be the multi kink-soliton profile defined in \eqref{multi kink soliton profile}. Since $R_j$ solves \eqref{dNLS}, for all $j$, by an elementary calculation, we have
\begin{equation}
\label{eq2-2new}
iV_t+V_{xx}+iV^2\overline{V_x}+b|V|^4V=i\left(V^2\overline{V_x}-\sum_{j=0}^{K}R_j^2\overline{R_{jx}}\right)+b\left(|V|^4V-\sum_{j=0}^{K}|R_j|^4R_j\right).
\end{equation}
From Lemma \ref{lm100}, we have
\begin{equation}
\label{eq2-3new}
\left\lVert V^2\overline{V_x}-\sum_{j=0}^{K}R_j^2\overline{R_{jx}}\right\rVert_{H^2}+\left\lVert|V|^4V-\sum_{j=0}^{K}|R_j|^4R_j\right\rVert_{H^2} \leq e^{-\lambda t}, 
\end{equation}
for $\lambda=\frac{1}{16}v_{*}$. Thus, we rewrite \eqref{eq2-2new} as follows
\begin{equation}
\label{eq2-4new}
iV_t+V_{xx}+iV^2\overline{V_x}+b|V|^4V = e^{-\lambda t}m(t,x),
\end{equation}
where $m(t) \in H^2(\R)$ such that $\norm{m(t)}_{H^2}$ uniformly bounded in $t$.
Define
\begin{align*}
h&=V,\\
k&=h_x+\frac{i}{2}|h|^2h.
\end{align*}
By an elementary calculation, $h,k$ satisfy the following system.
\begin{align*}
Lh&=-ih^2\overline{k}+\left(\frac{1}{2}- b\right)|h|^4h+e^{-t\lambda}m=P(h,k)+e^{-t\lambda}m,\\
Lk&=ik^2\overline{h}+\left(\frac{3}{2}-3b\right)|h|^4k+(1- 2b)|h|^2h^2\overline{k}+e^{-t\lambda}n=Q(h,k)+e^{-t\lambda}n.
\end{align*}
where $n=m_x+i|h|^2m-\frac{i}{2}h^2\overline{m}$ satisfies $\norm{n(t)}_{H^1}$ uniformly bounded in $t$. Let $\tilde u=u-h$ and $\tilde v=v-k$. Then $(\tilde u$, $\tilde v)$ solves:
\begin{equation}
\label{eq2-7new}
\begin{cases}
L\tilde u=P(u,v)-P(h,k)-e^{-t\lambda}m,\\
L\tilde v=Q(u,v)-Q(h,k)-e^{-t\lambda}n.
\end{cases}
\end{equation} 
Define $\eta=(\tilde u,\tilde v)$, $W=(h,k)$, $H=e^{-t\lambda}(m,n)$ and $f(u,v)=(P(u,v),Q(u,v))$. We find a solution of \eqref{eq2-7new} in the Duhamel form
\begin{equation}
\label{eq2-8new}
\eta=-i\int_t^{\infty}S(t-s)[f(W+\eta)-f(W)+H](s)\, ds.
\end{equation}
Moreover, from $v=u_x+\frac{i}{2}|u|^2u$, we have  
\begin{equation}
\label{eq2-9new}
\tilde v=\tilde u_x+\frac{i}{2}(|\tilde u+h|^2(\tilde u+h)-|h|^2h).
\end{equation}

\textbf{Step 2. Existence a solution of \eqref{eq2-8new}}

From Lemma \ref{lemma2}, there exists $T_{*} \gg 1$ such that for $T_0 \gg T_{*}$ there exists a unique solution $\eta$ defined on $[T_0,\infty)$ of \eqref{eq2-8new} such that 
\begin{equation}
\label{eq2-10new}
e^{t\lambda}\norm{\eta}_{S([t,\infty)) \times S([t,\infty))} +e^{t\lambda}\norm{\eta_x}_{S([t,\infty)) \times S([t,\infty))} \leq 1, \quad \forall t \geq T_0, 
\end{equation}
where $\lambda=\frac{v_*}{16}$. Thus, for all $t \geq T_0$, we have
\begin{equation}
\label{eq2-11new}
\norm{\tilde u}_{H^1}+\norm{\tilde v}_{H^1} \lesssim e^{-t\lambda}.
\end{equation}

\textbf{Step 3. Existence of multi kink-solitons}

By using similar arguments as in the proof of Theorem \ref{mainresult} we can prove that the solution $\eta=(\tilde\varphi,\tilde\psi)$ of \eqref{eq2-8new} satisfies the relation \eqref{eq2-9new} provided assumption \eqref{condition tecnique of multi kink soliton} is verified. This implies that 
\[
\tilde v=\tilde u_x+\frac{i}{2}(|\tilde u+h|^2(\tilde u+h)-|h|^2h). 
\]
Set $u=\tilde u+h$, $v=\tilde v+k$. We have
\begin{equation}\label{relation of solution in multi kink soliton case}
v=u_x+\frac{i}{2}|u|^2u.
\end{equation}
Since $(\tilde u,\tilde v)$ solves \eqref{eq2-7new}, we infer that $u,v$ solve
\begin{align*}
Lu&=P(u,v),\\
Lv &=Q(u,v).
\end{align*}
Combining with \eqref{relation of solution in multi kink soliton case}, we have $u$ solves \eqref{dNLS}. Moreover, for $t \geq T_0$, we have 
\begin{align*}
\norm{u-V}_{H^1} &= \norm{\tilde u}_{H^1}\lesssim e^{-\lambda t}.
\end{align*}
This completes the proof of Theorem \ref{mainresult2}. 

\begin{remark}\label{rm10}
We do not have the proof for the construction of multi kink-solitons for \eqref{dNLS}. The reason is that if the profile $R$ in the proof of Theorem \ref{mainresult} is not in $H^1(\R)$ then the function $h$ defined as in \eqref{eqof h} is not in $H^1(\R)$. Thus, the functions $m,n$ defined as in \eqref{expressionofmn1} and \eqref{expressionofmn2} are not in $H^1(\R)$ and we can not apply Lemma \ref{lemma2} to construct a solution of system \eqref{systemneedtosolve}.    
\end{remark}

\section{Some technical lemmas}
\label{chapter technique}

\subsection{Properties of solitons}
In this section, we prove some estimates on the multi-soliton profile used in the proof of Theorem \ref{mainresult}.

\begin{lemma}
\label{lemma1}
There exist $T_0>0$ and a constant $\lambda>0$ such that the estimate \eqref{eqestimate123} is uniformly true for $t \geq T_0$. 
\end{lemma}

\begin{proof}
First, we need some estimates on the soliton profile. We have 
\begin{align*}
|R_j(x,t)| &=|\Phi_{\omega_j,c_j}(x-c_j t)|=\sqrt{2}h_j\left(\sqrt{c_j^2+\gamma h_j^2}\cosh(h_j (x-c_j t))-c_j\right)^{-\frac{1}{2}} \\
&\quad\lesssim_{h_j,|c_j|} e^{\frac{-h_j}{2}|x-c_j t|}.  
\end{align*} 
Moreover,
\begin{align*}
|\partial R_j(x,t)| &= |\partial \phi_{\omega_j,c_j}(x-c_j t)| \\
&=\frac{-\sqrt{2}}{2} h_j^2 \sqrt{c_j^2+\gamma h_j^2}\left|\sinh(h_j (x-c_j t)\right|\left(\sqrt{c_j^2+\gamma h_j^2}\cosh(h_j (x-c_j t))-c_j\right)^{-\frac{3}{2}} \\
&\lesssim_{h_j,|c_j|} e^{\frac{-h_j}{2}|x-c_j t|}.
\end{align*}
By an elementary calculation, we have
\begin{align*}
|\partial^2 R_j(x,t)|+|\partial^3 R_j(x,t)| &\lesssim_{h_j,|c_j|} e^{\frac{-h_j}{2}|x-c_j t|}.
\end{align*}
For convenience, we set
\begin{align}
\chi_1&= i|R|^2R_x-i\sum_{j=1}^{K}|R_j|^2R_{jx},\label{define of chi 1}\\
\chi_2&=|R|^4R-\sum_{j=1}^{K}|R_j|^4R_j.\label{define of chi 2}
\end{align}
Fix $t>0$. For $x \in \R$, choose $m=m(x) \in \{1,2,...,K\}$ so that 
\[
|x-c_m t|=\mathop{\min}\limits_{j }|x-c_j t|.
\]
For $j \neq m$, we have
\[
|x-c_j t|\geq \frac{1}{2}|c_j t-c_m t|=\frac{t}{2}|c_j-c_m|.
\]
Thus, we have
\begin{align*}
&|(R-R_m)(x,t)|+|(\partial R-\partial R_{m}(x,t))|+|\partial^2 R-\partial^2R_m|+|\partial^3 R-\partial^3 R_m| \\
&\leq \sum_{j\neq m}(|R_j(x,t)|+|\partial R_{j}(x,t)|+|\partial^2 R_j(x,t)|+|\partial^3 R_j(x,t)|)\\
&\lesssim_{h_1,..,h_K,|c_1|,..,|c_K|}  \delta_m(x,t):= \sum_{j\neq m}  e^{\frac{-h_j}{2}|x-c_jt|}
\end{align*}
Recall that
\[
v_{*}=\mathop{\inf}\limits_{j \neq k}h_j|c_j-c_k|.
\]
We have
\[
|(R-R_m)(x,t)|+|(\partial R-\partial R_{m}(x,t))|+|\partial^2 R-\partial^2R_m|+|\partial^3 R-\partial^3 R_m| \lesssim  \delta_m(x,t) \lesssim e^{\frac{-1}{4}v_{*}t}.
\]
Let $f_1,g_1,r_1$ and $f_2,g_2,r_2$ be the polynomials of $u,u_x,u_{xx},u_{xxx}$ and conjugates satisfying: 
\begin{align*}
i|u|^2u_x=f_1(u,\overline{u},u_x), & \quad |u|^4u=f_2(u,\overline{u}),\\
\partial(i|u|^2u_x)=g_1(u,u_x,u_{xx},\overline{u},..),& \quad \partial(|u|^4u)=g_2(u,u_x,\overline{u},..),\\
\partial^2(i|u|^2u_x)=r_1(u,u_x,u_{xx},u_{xxx},\overline{u},..),&\quad \partial^2(|u|^4u)=r_2(u,u_x,u_{xx},\overline{u},..).
\end{align*} 
Denote 
\begin{align*}
A=\mathop{\sup}\limits_{|u|+|u_x|+|u_{xx}|+|u_{xxx}|\leq \mathop{\sum}\limits_{j=1}^{K}\norm{R_j}_{H^4}}&(|df_1|+|df_2|+|dg_1|+|dg_2|+|dr_1|+|dr_2|),
\end{align*}
We have
\begin{align*}
&|\chi_1|+|\chi_2|+|\partial \chi_1|+|\partial \chi_2| +|\partial^2 \chi_1|+|\partial^2 \chi_2|\\
&\leq |f_1(R,R_{x})-f_1(R_m,R_{mx})|+|f_2(R)-f_2(R_m)|+\sum_{j\neq m}(|f_1(R_j,R_{jx})|+|f_2(R_j)|)\\
&\quad +|g_1(R,R_x,R_{xx},..)-g_1(R_m,R_{mx},R_{mxx},..)|+|g_2(R,R_x,..)-g_2(R_m,R_{mx},..)|\\
&+\sum_{j\neq m}(g_1(R_j,R_{jx},R_{jxx},..)+g_2(R_j,R_{jx}),..)\\
&\quad + |r_1(R,R_x,R_{xx},R_{xxx},..)-r_1(R_m,R_{mx},R_{mxx},R_{mxxx},..)|\\
&\quad +|r_2(R,R_x,R_{xx},..)-r_2(R_m,R_{mx},R_{mxx},..)|\\
&\quad +\sum_{j\neq m}(r_1(R_j,R_{jx},R_{jxx},R_{jxxx},..)+r_2(R_j,R_{jx},R_{jxx},..))\\
&\leq A(|R-R_m|+|R_x-R_{mx}|+|R_{xx}-R_{mxx}|+|R_{xxx}-R_{mxxx}|)\\
&\quad+\sum_{j\neq m}A(|R_j|+|R_{jx}|+|R_{jxx}|+|R_{jxxx}|)\\
& \leq 2A\sum_{j\neq m}(|R_j|+|R_{jx}|+|R_{jxx}|+|R_{jxxx}|)\\
&\lesssim_p \delta_m(t,x).
\end{align*}
In particular,
\[
\norm{\chi_1}_{W^{2,\infty}}+\norm{\chi_2}_{W^{2,\infty}}\lesssim_p e^{-\frac{1}{4}v_{*}t}. 
\]
Moreover, we have
\begin{align*}
&\norm{\chi_1}_{W^{2,1}}+\norm{\chi_2}_{W^{2,1}} \\
&\lesssim \sum_{j=1}^K (\norm{|R_j|^2R_{jx}}_{L^1}+\norm{\partial (|R_j|^2R_{jx})}_{L^1}+\norm{\partial^2(|R_j|^2R_{jx})}_{L^1}\\
&\quad+\norm{R_j^5}_{L^1}+\norm{\partial(|R_j|^4R_j)}_{L^1}+\norm{\partial^2(|R_j|^4R_j)}_{L^1})\\
&\lesssim  \sum_{j=1}^K(\norm{R_j}^3_{H^1}+\norm{R_j}^3_{H^2}+\norm{R_j}^3_{H^3}+\norm{R_j}^5_{H^1}+\norm{R_j}^5_{H^1}+\norm{R_j}^5_{H^2})<C<\infty
\end{align*}
By Holder inequality, for $1<r<\infty$, we have
\[
\norm{\chi_1}_{W^{2,r}}+\norm{\chi_2}_{W^{2,r}}\lesssim_p e^{-(1-\frac{1}{r})\frac{1}{4} v_{*}t}, \quad \forall r \in (1,\infty).
\]
Choosing $r=2$ we obtain:
\[
\norm{\chi_1}_{H^2}+\norm{\chi_2}_{H^2} \lesssim_p e^{-\frac{v_{*}}{8}t},
\]
Thus, for $t \geq T_0$, where $T_0$ large enough depend on the parameters $\omega_1,...,\omega_K,c_1,...,c_K$, we have
\[
\norm{\chi_1}_{H^2}+\norm{\chi_2}_{H^2}\leq e^{-\frac{v_{*}}{16}t} ,\quad \forall t \geq T_0.
\] 
Let $\lambda=\frac{v_{*}}{16}$, we obtain the desired result. 
\end{proof}

\subsection{Prove the boundedness of $v,m,n$}
Let $v$, $m$ and $n$ be given as in \eqref{eqR}, \eqref{expressionofmn1} and \eqref{expressionofmn2} respectively. In this section, we prove the uniform in time boundedness in $H^2(\R)$ of $v$ and in $H^1(\R)$ of $m,n$. We have the following result.
\begin{lemma}
\label{lm4}
There exist $C>0$ and $T_0>0$ such that for all $t>T_0$ the functions $v,m,n$ satisfy
\[
\norm{v(t)}_{H^2}+\norm{m(t)}_{H^1}+\norm{n(t)}_{H^1} \leq C,
\]
\end{lemma}

\begin{proof}
Let $\chi_1$ and $\chi_2$ be defined as in \eqref{define of chi 1} and \eqref{define of chi 2} respectively. We have
\[
e^{-\lambda t} v =\chi_1 +b\chi_2.
\]
By Lemma \ref{lemma1}, we have $\norm{v(t)}_{H^2} \leq D$, for some constant $D>0$. From \eqref{expressionofmn1}, we have
\begin{align*}
\norm{m}_{H^2} \lesssim \norm{v}_{H^2}+\norm{h}_{H^2}\norm{v}_{H^2}\norm{R}_{H^2} \leq C_1,
\end{align*} 
for some constant $C_1>0$. From, \eqref{expressionofmn2}, we have
\begin{align*}
\norm{n}_{L^2} \lesssim \norm{m_x}_{L^2} +\norm{h}^2_{H^1}\norm{m}_{H^1} \leq \norm{m}_{H^1}(1+\norm{h}^2_{H^1}) \leq C_2,
\end{align*}
for some constant $C_2>0$. Moreover, we have
\begin{align*}
\norm{n_x}_{L^2} \lesssim \norm{m_{xx}}_{L^2}+\norm{h}^2_{H^1}\norm{m}_{H^1} \leq \norm{m}_{H^2}(1+\norm{h}^2_{H^1}) \leq C_3,
\end{align*}
for some constant $C_3>0$. Choosing $C=D+C_1+C_2+C_3$, we obtain the desired result.
\end{proof}

\subsection{Existence solution of system equation}

In this section, we prove the existence of solutions of \eqref{eqof eta}. For convenience, we recall the equation:
\begin{equation}
\label{eqgeneralneedtosolve}
\eta(t)=i\int_t^{\infty}S(t-s)[f(W+\eta)-f(W)+H](s)\, ds,
\end{equation}
where $\eta=(\tilde{u},\tilde{v})$ is unknown function, $W=(h,k)$, $H=-e^{-t\lambda}(m,n)$ and $f(u,v)=(P(u,v),Q(u,v))$, where $P,Q$ are defined by
\begin{align*}
P(u,v)&=-iu^2\overline{v}+\left(\frac{1}{2}-b\right)|u|^4u,\\
Q(u,v)&=iv^2\overline{u}+\left(\frac{3}{2}-3b\right)|u|^4v+(1-2b)|u|^2u^2\overline{v}.
\end{align*}
The existence of solutions of \eqref{eqgeneralneedtosolve} is established in the following lemma.
\begin{lemma}
\label{lemma2}
Let $H=H(t,x):[0,\infty)\times \R \rightarrow \C^2$, $W=W(t,x):[0,\infty)\times \R \rightarrow \C^2$ be given vector functions which satisfy for some $C_1>0$, $C_2>0$, $\lambda>0$, $T_0 \geq 0$:
\begin{align}
\norm{W(t)}_{L^{\infty}\times L^{\infty}}+e^{\lambda t}\norm{H(t)}_{L^2 \times L^2} &\leq C_1 \quad \forall t \geq T_0, \label{estimateWHinL2norma1}\\
\norm{\partial W(t)}_{L^2 \times L^2}+\norm{\partial W(t)}_{L^{\infty} \times L^{\infty}}+e^{\lambda t}\norm{\partial H(t)}_{L^2 \times L^2} &\leq C_2, \quad \forall t \geq T_0. \label{estimateWHinderivativenorma1}
\end{align} 
Consider equation \eqref{eqgeneralneedtosolve}. There exists a constant $\lambda_{*}$ such that if $\lambda \geq \lambda_{*}$ then there exists a unique solution $\eta$ to \eqref{eqgeneralneedtosolve} on $[T_0,\infty) \times \R$ satisfying
\[
e^{\lambda t}\norm{\eta}_{S([t,\infty)) \times S([t,\infty))}+e^{\lambda t}\norm{\partial \eta}_{S([t,\infty)) \times S([t,\infty))} \leq 1, \quad \forall t \geq T_0.
\]
\end{lemma}

\begin{proof}
We use similar arguments as in \cite{CoDoTs15,CoTs14}. We rewrite \eqref{eqgeneralneedtosolve} into $\eta=\Phi\eta$. We shall show that, for $\lambda$ sufficiently large, $\Phi$ is a contraction map in the ball
\[
B=\left\{\eta:\norm{\eta}_{X}:=e^{\lambda t}\norm{\eta}_{S([t,\infty)) \times S([t,\infty))}+e^{\lambda t}\norm{\partial\eta}_{S([t,\infty)) \times S([t,\infty))} \leq 1 \right\}.
\]

\textbf{Step 1. Proof that $\Phi$ maps $B$ into $B$}

Let $t \geq T_0$, $\eta=(\eta_1,\eta_2) \in B$, $W=(w_1,w_2)$ and $H=(h_1,h_2)$. By Strichartz estimates, we have
\begin{align}
\norm{\Phi\eta}_{S([t,\infty)) \times S([t,\infty))} &\lesssim \norm{f(W+\eta)-f(W)}_{N([t,\infty)) \times N([t,\infty))} \label{term1a}\\
&\quad+\norm{H}_{L_{\tau}^1L_x^2([t,\infty)) \times L_{\tau}^1L_x^2([t,\infty))}.\label{term2a}
\end{align} 
For \eqref{term2a}, using \eqref{estimateWHinL2norma1}, we have
\begin{align*}
\norm{H}_{L_{\tau}^1L_x^2([t,\infty)) \times L_{\tau}^1L_x^2([t,\infty))} &= \norm{h_1}_{L_{\tau}^1L_x^2([t,\infty))}+\norm{h_2}_{L_{\tau}^1L_x^2([t,\infty))}\\
&\quad\lesssim \int_t^{\infty}e^{-\lambda \tau} \, d\tau \leq \frac{1}{\lambda} e^{-\lambda t}. 
\end{align*}
For \eqref{term1a}, we have
\begin{align*}
&|P(W+\eta)-P(W)| = |P(w_1+\eta_1,w_2+\eta_2)-P(w_1,w_2)|\\
&\lesssim |(w_1+\eta_1)^2\overline{(w_2+\eta_2)}-w_1^2\overline{w_2}| + ||\eta_1+w_1|^4(\eta_1+w_1)-|w_1|^4w_1|\\
&\lesssim |\eta_1|+|\eta_2|+|\eta_1|^5  
\end{align*}
Thus,
\begin{align*}
\norm{P(W+\eta)-P(W)}_{N([t,\infty))} &\lesssim \norm{\eta_1}_{N([t,\infty))}+\norm{\eta_2}_{N([t,\infty))}+\norm{\eta_1^5}_{N([t,\infty))}\\
&\lesssim \norm{\eta_1}_{L^1_{\tau}L^2_x(t,\infty)}+\norm{\eta_2}_{L^1_{\tau}L^2_x(t,\infty)}+\norm{\eta_1^5}_{L^1_{\tau}L^2_x(t,\infty)} \\
&\lesssim \int_t^{\infty}e^{-\lambda \tau}\, d\tau + \int_t^{\infty} \norm{\eta_1(\tau)}^5_{L^{10}}\, d\tau \\
&\lesssim  \frac{1}{\lambda}e^{-\lambda t} + \int_t^{\infty}\norm{\eta_1(\tau)}_{L^2}^{\frac{7}{2}}\norm{\partial\eta_1(\tau)}_{L^2}^{\frac{3}{2}} \\
&\lesssim  \frac{1}{\lambda}e^{-\lambda t} + \int_t^{\infty}e^{-(7/2 \lambda + 3/2 \lambda)\tau}\, d\tau\\
&\lesssim \frac{1}{\lambda}e^{-\lambda t} + \frac{1}{7/2 \lambda +3/2 \lambda}e^{-(7/2 \lambda +3/2\lambda)t} \lesssim \frac{1}{\lambda}e^{-\lambda t}.
\end{align*}
By similar arguments as above, we have
\[
\norm{Q(W+\eta)-Q(W)}_{N([t,\infty))} \lesssim \frac{1}{\lambda} e^{-\lambda t}.
\]
Thus, for $\lambda$ large enough, we have
\[
\norm{\Phi\eta}_{S([t,\infty) \times S([t,\infty)))} \leq \frac{1}{10} e^{-\lambda t}.
\]
It remains to estimate $\norm{\partial\Phi\eta}_{S([t,\infty) \times S([t,\infty)))}$. By Strichartz estimate we have
\begin{align}
\norm{\partial\Phi\eta}_{S([t,\infty) \times S([t,\infty)))} &\lesssim \norm{\partial(f(W+\eta)-f(W))}_{N([t,\infty)) \times N([t,\infty))} \label{term3a}\\
&\quad+\norm{\partial H}_{N([t,\infty)) \times N([t,\infty))}.\label{term4a}
\end{align}
For \eqref{term4a}, using \eqref{estimateWHinderivativenorma1}, we have
\begin{align}
\norm{\partial H}_{N([t,\infty)) \times N([t,\infty))} &\leq \norm{\partial h_1}_{L^1_{\tau}L^2_x([t,\infty))}+\norm{\partial h_2}_{L^1_{\tau}L^2_x([t,\infty))} \nonumber\\
& \lesssim \int_t^{\infty} e^{-\lambda \tau}\, d\tau =\frac{1}{\lambda}e^{-\lambda t}. \label{superbigterm1}
\end{align}
For \eqref{term3a}, we have
\begin{align*}
&\norm{\partial (f(W+\eta)-f(W))}_{N([t,\infty)) \times N([t,\infty))} \\
& = \norm{\partial (P(W+\eta)-P(W))}_{N([t,\infty))}+\norm{\partial (Q(W+\eta)-Q(W))}_{N([t,\infty))}
\end{align*}
Furthermore,
\begin{align*}
&|\partial(P(W+\eta)-P(W))|\\
&\lesssim |\partial((w_1+\eta_1)^2\overline{(w_2+\eta_2)}-w_1^2\overline{w_2})|+|\partial (|w_1+\eta_1|^4(w_1+\eta_1)-|w_1|^4w_1)|\\
&\lesssim |\partial\eta|(|\eta|^2+|W|^2)+|\partial W|(|\eta|^2+|W||\eta|) \\
&\quad + |\partial\eta|(|\eta|^4+|W|^4)+|\partial W|(|\eta|^4+|\eta||W|^3).
\end{align*}
Thus, we have
\begin{align}
&\norm{\partial (P(W+\eta)-P(W))}_{N([t,\infty))}
\nonumber\\ &\lesssim \norm{|\partial\eta|(|\eta|^2+|W|^2)}_{N([t,\infty))}+\norm{|\partial W|(|\eta|^2+|W||\eta|)}_{N([t,\infty))}\label{term5}\\
&\quad +\norm{|\partial\eta|(|\eta|^4+|W|^4)}_{N([t,\infty))}+\norm{|\partial W|(|\eta|^4+|\eta||W|^3)}_{N([t,\infty))}.\label{term6}
\end{align}
For \eqref{term5}, using \eqref{estimateWHinL2norma1} and \eqref{estimateWHinderivativenorma1} and the assumption $\eta \in B$ we have
\begin{align*}
&\norm{|\partial\eta|(|\eta|^2+|W|^2)}_{N([t,\infty))}+\norm{|\partial W|(|\eta|^2+|W||\eta|)}_{N([t,\infty))}\\
&\lesssim \norm{|\partial\eta||\eta|^2}_{L^1_{\tau}L^2_x([t,\infty))}+\norm{|\partial \eta||W|^2}_{L^1_{\tau}L^2_x([t,\infty))}+\norm{|\partial W||\eta|^2}_{L^1_{\tau}L^2_x([t,\infty))}\\
&\quad+\norm{|\partial W||W||\eta|}_{L^1_{\tau}L^2_x([t,\infty))}\\
&\lesssim \norm{|\partial\eta|}_{L^2_{\tau}L^2_x([t,\infty))}\norm{|\eta|}^2_{L^4_{\tau}L^{\infty}}+\norm{|\partial\eta|}_{L^1_{\tau}L^2_x([t,\infty))}\norm{|W|}^2_{L^{\infty}L^{\infty}}\\
&\quad +\norm{|\partial W|}_{L^{\infty}L^{\infty}}\norm{|\eta|}_{L^4_{\tau}L^{\infty}_x([t,\infty))}\norm{|\eta|}_{L^{4/3}_{\tau}L^2_x([t,\infty))}\\
&\quad+\norm{|W|}_{L^{\infty}L^{\infty}}\norm{|\partial W|}_{L^{\infty}L^{\infty}}\norm{|\eta|}_{L^1_{\tau}L^2_x([t,\infty))}\\
&\lesssim \frac{1}{\lambda}e^{-\lambda t}.
\end{align*}

For \eqref{term6}, using \eqref{estimateWHinL2norma1} and \eqref{estimateWHinderivativenorma1} and the assumption $\eta \in B$ we have

\begin{align*}
&\norm{|\partial\eta|(|\eta|^4+|W|^4)}_{N([t,\infty))}+\norm{|\partial W|(|\eta|^4+|\eta||W|^3)}_{N([t,\infty))}\\
&\lesssim \norm{|\partial\eta|(|\eta|^4+|W|^4)}_{L^1_{\tau}L^2_x([t,\infty))}+\norm{|\partial W|(|\eta|^4+|\eta||W|^3)}_{L^1_{\tau}L^2_x([t,\infty))}\\
&\lesssim \norm{\partial\eta}_{L^{\infty}_{\tau}L^2_x([t,\infty))}\norm{\eta}^4_{L^4_{\tau}L^{\infty}_x([t,\infty))}+\norm{W}^4_{L^{\infty}L^{\infty}}\norm{\partial\eta}_{L^1_{\tau}L^2_x([t,\infty))}\\
&\quad  +\norm{\partial W}_{L^{\infty}L^2}\norm{\eta}^4_{L^4_{\tau}L^{\infty}_x([t,\infty))} + \norm{\partial W}_{L^{\infty}L^{\infty}}\norm{|W|}^3_{L^{\infty}L^{\infty}}\norm{\eta}_{L^1_{\tau}L^2_x([t,\infty))}\\
&\lesssim \frac{1}{\lambda}e^{-\lambda t}.
\end{align*}
Hence,
\begin{align}\label{bigterm1}
\norm{\partial(P(W+\eta)-P(W))}_{N([t,\infty))} \lesssim \frac{1}{\lambda}e^{-\lambda t}.
\end{align}
By similar arguments, we have
\begin{align}\label{bigterm2}
\norm{\partial(Q(W+\eta)-Q(W))}_{N([t,\infty))} \lesssim \frac{1}{\lambda}e^{-\lambda t}.
\end{align}
Combining \eqref{bigterm1} and \eqref{bigterm2}, we obtain
\begin{equation}
\label{superbigterm2}
\norm{\partial(f(W+\eta)-f(W))}_{N([t,\infty)) \times N([t,\infty))} \lesssim \frac{1}{\lambda}e^{-\lambda t}.
\end{equation}
Combining \eqref{superbigterm1} and \eqref{superbigterm2}, we obtain
\begin{equation*}
\norm{\partial\Phi\eta}_{S([t,\infty)) \times S([t,\infty))}\lesssim \frac{1}{\lambda}e^{-\lambda t} \leq \frac{1}{10}e^{-\lambda t},
\end{equation*}
if $\lambda>0$ is large enough. Thus, for $\lambda>0$ large enough
\begin{equation}
\label{finalestimate}
\norm{\Phi\eta}_{X} \leq 1.
\end{equation} 
This implies that $\Phi$ map $B$ onto $B$.

\textbf{Step 2. $\Phi$ is contraction map on $B$}

By using \eqref{estimateWHinL2norma1} and \eqref{estimateWHinderivativenorma1} and similar estimates as for the proof of \eqref{finalestimate}, we can show that, for any $\eta \in B$, $\kappa \in B$,
\[
\norm{\Phi\eta - \Phi\kappa}_{X} \leq \frac{1}{2}\norm{\eta-\kappa}_{X}.
\] 
By Banach fixed point theorem there exists a unique solution on $B$ of \eqref{eqgeneralneedtosolve}.
\end{proof}

\subsection{Properties of multi kink-solitons profile}
In this section, we prove some estimates on the multi kink-solitons profile used in the proof of Theorem \ref{mainresult2}. 
\begin{lemma}
\label{lm100}
There exist $T_0>0$ and a constant $\lambda>0$ such that the estimate \eqref{eq2-3new} is uniformly true for $t \geq T_0$. 
\end{lemma}

\begin{proof}
For convenience, set 
\[
R=\sum_{j=1}^K R_j.
\]
By similar arguments in the proof of Lemma \ref{lemma1}, we have
\begin{align*}
|R_j(x,t)|+|\partial R_j(x,t)|+|\partial^2 R_j(x,t)|+|\partial^3 R_j(x,t)| &\lesssim_{h_j,|c_j|} e^{\frac{-h_j}{2}|x-c_j t|},
\end{align*}
for all $1 \leq j \leq K$. Define
\begin{align*}
\chi_1&= iV^2\overline{V_x}-i\sum_{j=0}^{K}R_j^2\overline{R_{jx}},\\
\chi_2&=|V|^4V-\sum_{j=0}^{K}|R_j|^4R_j.
\end{align*}
Fix $t>0$. For $x \in \R$, we choose $m=m(x) \in \N$ such that
\[
|x-c_m t| =\mathop{\min}\limits_{j\in\N}|x-c_j t|.
\]
If $m \geq 1$ then by the assumption $c_0<c_j$ for $j>0$ we have $x > c_0 t$. Thus, by the asymptotic behaviour of $\Phi_0$ as in Remark \ref{remark constant a}, we can see $R_0$ as a soliton. More precise, we have
\[
|R_0(t,x)|+|R_0'(t,x)|+|R_0''(t,x)|+|R_0'''(t,x)| \lesssim  e^{-\frac{1}{2}|x-c_0 t|} \lesssim e^{-\frac{1}{4}v_{*}t}.
\]   
Using similar argument as in the proof of Lemma \ref{lemma1}, we have:
\[
|(R-R_m)(x,t)|+|(\partial R-\partial R_m)(x,t)|+|(\partial^2 R-\partial^2 R_m)(x,t)|+|\partial^3 R-\partial^3 R_m| \lesssim e^{-\frac{1}{4}v_{*}t}.
\]
Let $f_1,g_1,r_1$ and $f_2,g_2,r_2$ be the polynomials of $u,u_x,u_{xx},u_{xxx}$ and their conjugates such that for all $u \in H^3(\R)$: 
\begin{align*}
iu^2\overline{u_x}=f_1(u,\overline{u},u_x), & \quad |u|^4u=f_2(u,\overline{u}),\\
\partial(iu^2\overline{u_x})=g_1(u,u_x,u_{xx},\overline{u},..),& \quad \partial(|u|^4u)=g_2(u,u_x,\overline{u},..),\\
\partial^2(iu^2\overline{u_x})=r_1(u,u_x,u_{xx},u_{xxx},\overline{u},..),&\quad \partial^2(|u|^4u)=r_2(u,u_x,u_{xx},\overline{u},..).
\end{align*} 
Denote 
\[
A=\mathop{\sup}\limits_{|u|+|u_x|+|u_{xx}|+|u_{xxx}|\leq \norm{R_0}_{W^{4,\infty}}+\sum_{j=1}^K \norm{R_j}_{H^4(\R)}}(|df_1|+|df_2|+|dg_1|+|dg_2|+|dr_1|+|dr_2|).
\]
In the case $m=1$, we have
\begin{align*}
&|\chi_1|+|\chi_2|+|\partial \chi_1|+|\partial \chi_2| +|\partial^2 \chi_1|+|\partial^2 \chi_2|\\
&\lesssim |R_0|^2|R_{0x}|+|R_0|^5+|f_1(V,..)-f_1(R,..)|+|f_2(V,...)-f_2(R,..)|\\
&\quad+|g_1(V,..)-g_1(R,..)|+|g_2(V,..)-g_2(R,..)|\\
&\quad+|r_1(V,..)-r_1(R,..)|+|r_2(V,..)-r_2(R,..)|+|f_1(R,R_x,\overline{R})-\sum_{j=1}^K f_1(R_j,R_{jx},\overline{R_j})|\\
&\quad+|f_2(R,\overline{R})-\sum_{j=0}^Kf_2(R_j,\overline{R_j})| +|g_1(R,R_x,..)-\sum_{j=1}^K g_1(R_j,R_{jx},..)|\\
&\quad+|g_2(R,R_x,..)-\sum_{j=0}^Kg_2(R_j,R_{jx},..)|+|r_1(R,R_x,..)-\sum_{j=0}^Kr_1(R_j,R_{jx},..)|\\
&\quad+|r_2(R,R_x,..)-\sum_{j=0}^Kr_2(R_j,R_{jx},..)|\\
&\lesssim |R_0|^2|R_{0x}|+|R_0|^5+A |R_0|\\
&\quad +A(|(R-R_m)(x,t)|+|(\partial R-\partial R_m)(x,t)|+|(\partial^2 R-\partial^2 R_m)(x,t)|+|\partial^3 R-\partial^3 R_m|)\\
&\quad +A\sum_{j=1,j\neq m}^K(|R_j|+|\partial R_j|+|\partial^2R_j|+|\partial^3R_j|)\\
& \lesssim |R_0|^2|R_{0x}|+|R_0|^5+A |R_0|+A\sum_{j=1,j\neq m}^K(|R_j|+|\partial R_j|+|\partial^2R_j|+|\partial^3R_j|)\\
&\lesssim_p e^{-\frac{1}{4}v_{*}t}, 
\end{align*}
In the case $m=0$, we have 
\begin{align*}
&|\chi_1|+|\chi_2|+|\partial \chi_1|+|\partial \chi_2| +|\partial^2 \chi_1|+|\partial^2 \chi_2|\\
&\lesssim \sum_{v=1,2} (|f_v(V,V_x,..)-f_v(R_0,\partial R_0,..)|+|g_v(V,V_x,..)-g_v(R_0,\partial R_0,..)|\\
&\quad+|r_v(V,V_x,..)-r_v(R_0,\partial R_0)|)\\
&\quad + \sum_{j=1,...,K;v=1,2} (|f_v(R_j,R_{jx},..)|+|g_v(R_j,R_{jx},..)|+|r_v(R_j,R_{jx},..)|)\\
&\lesssim A|R|+A\sum_{j=1}^K (|R_j|+|\partial R_j|+|\partial^2 R_j|+|\partial^3 R_j|)\\
& \lesssim_p e^{-\frac{1}{4}v_{*}t}.
\end{align*}
In all case we have  
\begin{equation}\label{qq}
\norm{\chi_1(t)}_{W^{2,\infty}}+\norm{\chi_2(t)}_{W^{2,\infty}}\lesssim_p e^{-\frac{1}{4}v_{*}t}.
\end{equation}
On one hand,
\begin{align*}
&\norm{\chi_1(t)}_{W^{2,1}} \\
&\lesssim \sum_{j=0}^K (\norm{R_j^2\overline{R_{jx}}}_{L^1}+\norm{\partial (R_j^2\overline{R_{jx}})}_{L^1}+\norm{\partial^2(R_j^2\overline{R_{jx}})}_{L^1})\\
&\lesssim  \sum_{j=1}^K\norm{R_j}^3_{H^3}+\norm{\partial R_0}_{W^{2,1}}<C<\infty.
\end{align*}
On the other hand,
\begin{align*}
&\norm{\chi_2(t)}_{W^{2,1}}\\
&\lesssim \norm{|V|^4V-|R_0|^4R_0}_{W^{2,1}}+\sum_{j=1}^K\norm{|R_j|^5}_{W^{2,1}}\\
&\lesssim \left\lVert|R_0|^4\sum_{j=1}^K|R_j|+\sum_{j=1}^K|R_j|^5\right\rVert_{W^{2,1}}+\sum_{j=1}^K\norm{R_j}^5_{W^{2,1}}\\
&\lesssim \sum_{j=1}^K \norm{|R_0|^4|R_j|}_{W^{2,1}}+\sum_{j=1}^K\norm{R_j}^5_{W^{2,1}}\\
&\lesssim \sum_{j=1}^K(\norm{R_j}_{W^{2,1}}\norm{R_0}^4_{W^{2,\infty}}+\norm{R_j}^5_{H^3})<C<\infty.
\end{align*}
Thus, 
\begin{equation}\label{qw}
\norm{\chi_1(t)}_{W^{2,1}}+\norm{\chi_1(t)}_{W^{2,1}}<\infty.
\end{equation}
From \eqref{qq} and \eqref{qw}, using H\"older inequality, we have
\[
\norm{\chi_1(t)}_{H^2}+\norm{\chi_2(t)}_{H^2} \lesssim_p e^{-\frac{1}{8}v_{*}t}.
\]
Let $T_0$ be large enough, we have
\[
\norm{\chi_1(t)}_{H^2}+\norm{\chi_2(t)}_{H^2} \leq e^{-\frac{1}{16}v_{*}t}, \quad \forall t \geq T_0.
\] 
Setting $\lambda=\frac{1}{16}v_{*}$, we obtain the desired result.
\end{proof}

\section*{Acknowledgement} I wishes to thank Prof.Stefan Le Coz for his guidance and encouragement. I am supported by scholarship of MESR for his phD. This work is also supported by the ANR LabEx CIMI (grant ANR-11-LABX-0040) within the French State Programme “Investissements d’Avenir. Finally, I wishes to thank the unknown referees for careful reading and many useful discussions to improve this paper.   


\bibliographystyle{abbrv}
\bibliography{bibliothequemultisolitonmultikink}

\end{document}